\documentclass{amsart}
\usepackage{hyperref}
\usepackage{amssymb}
\usepackage{amsmath}
\usepackage{amsthm}
\usepackage{amssymb}
\usepackage[all]{xy}
\title[Character Sheaves over Non-Archimedean Local Fields]{Character Sheaves of Algebraic Groups Defined over Non-Archimedean Local Fields}

\author{Clifton~Cunningham}
\address{Department of Mathematics, University of Calgary, Canada}
\email{cunning@math.ucalgary.ca} 

\author{Hadi Salmasian}
\address{Department of Mathematics, University of Alberta, Canada}
\email{salmasia@ualberta.ca} 

\date{\today}
\newtheorem{theorem}{Theorem}
\newtheorem{proposition}{Proposition}
\newtheorem{corollary}{Corollary}
\newtheorem{lemma}{Lemma}
\theoremstyle{definition}
\newtheorem{definition}{Definition}
\theoremstyle{remark}
\newtheorem{remark}{Remark}

\renewcommand{\AA}{{\mathbb{A}}}
\newcommand{\NN}{{\mathbb{N}}}
\newcommand{\ZZ}{{\mathbb{Z}}}
\newcommand{\QQ}{{\mathbb{Q}}}

\newcommand{\CC}{{\mathbb{C}}}
\newcommand{\EE}{{\bar{\QQ}_\ell}}

\newcommand{\FF}{{\mathbb{F}}}

\newcommand{\KKq}{{\bar{\Kq}}}
\newcommand{\tKq}{{\Kq^{\operatorname{tr}}}}

\newcommand{\RKK}{{\mathfrak{o}_{\KKq}}}
\newcommand{\kkq}{{\bar{\kq}}}

\newcommand{\Kq}{{\mathbb{K}}}
\newcommand{\Rq}{{\mathfrak{o}_{\Kq}}}

\newcommand{\kq}{{\Bbbk}}

\newcommand{\sch}[1]{{#1}}
\newcommand{\Ksch}[1]{{{#1}}}
\newcommand{\KKsch}[1]{{{#1}_{\KKq}}}
\newcommand{\Rsch}[1]{{\underline{#1}}}

\newcommand{\ksch}[1]{{{#1}}}
\newcommand{\kksch}[1]{{{#1}_{\kkq}}}

\newcommand{\Sch}[1]{{\underline{#1}}}
\newcommand{\RSch}[2]{{\underline{#1}_{#2}}}
\newcommand{\RRSch}[2]{{ {\underline{\bar #1}}_{#2} }}
\newcommand{\KSch}[2]{{(\underline{#1}_{#2})_{\eta}}}
\newcommand{\KKSch}[2]{{(\underline{#1}_{#2})_{\bar \eta}}}
\newcommand{\kSch}[2]{{(\underline{#1}_{#2})_{{\bf s}}}}
\newcommand{\kkSch}[2]{{ (\underline{#1}_{#2})_{\bar{\bf s}}}}
\newcommand{\rkSch}[2]{{(\underline{#1}_{#2})_{{\bf s}}^{\operatorname{red}} }}
\newcommand{\rkkSch}[2]{{(\underline{#1}_{#2})_{\bar{\bf s}}^{\operatorname{red}} }}

\newcommand{\Lq}{{\Kq'}}
\newcommand{\RLq}{{\mathfrak{o}_{\Kq'}}}

\newcommand{\LBsch}[1]{{\Ksch{#1}}}
\newcommand{\LLBsch}[1]{{\KKsch{#1}}}

\newcommand{\RLBSch}[2]{{ \underline{#1}_{#2} }}
\newcommand{\RRLBSch}[2]{{ {\underline{\bar #1}}_{#2} }}
\newcommand{\LBSch}[2]{{ (\underline{#1}_{#2})_{\eta'}}}
\newcommand{\LLBSch}[2]{{ (\underline{#1}_{#2})_{\bar{\eta}}}}
\newcommand{\lBSch}[2]{{ (\underline{#1}_{#2})_{{\bf s}'}}}
\newcommand{\llBSch}[2]{{ (\underline{#1}_{#2})_{\bar{\bf s}} }}

\newcommand{\Lsch}[1]{{\Ksch{#1}_{\Lq}}}

\newcommand{\RLSch}[2]{{{\underline{#1}_{#2}}_{\RLq} }}

\newcommand{\llSch}[2]{{ (\underline{#1}_{#2})_{\bar{\bf s}} }}

\newcommand{\RESFAIS}[3]{{\RES{#2}{#3}\hskip-2pt\fais{#1}}}

\newcommand{\fais}[1]{{ \mathcal{#1} }}

\newcommand{\kPsch}[2]{{ {#1}_{#2} }}
\newcommand{\kkPsch}[2]{{ ({#1}_{#2})_{\bar{\bf s}} }}
\newcommand{\rkPsch}[2]{{ \sch{P}_{#2} }}
\newcommand{\rkkPsch}[2]{{ (\sch{P}_{#2})_{\bar{\bf s}} }}

\def\cf{{{\it cf.\ }}}
\def\ie{{{\it i.e.,\ }}}

\newcommand{\tq}{{\ \vert\ }}
\newcommand{\iso}{{\ \cong\ }}
\newcommand{\Hom}{{\operatorname{Hom}}}
\newcommand{\ceq}{{\, := \, }}

\newcommand{\res}{{\operatorname{res}}}
\newcommand{\ind}{{\operatorname{ind}}}
\newcommand{\Spec}[1]{{\operatorname{Spec}(#1)}}
\newcommand{\Ind}{{\operatorname{Ind}}}
\newcommand{\Res}{{\operatorname{Res}}}

\newcommand{\cRes}{{\operatorname{cRes}}}
\newcommand{\catM}{{\mathcal{M}}}
\newcommand{\proj}{{\operatorname{pr}}}

\newcommand{\id}{{\, \operatorname{id}}}
\newcommand{\Frob}{{\operatorname{F\hskip-1pt rob}}}
\newcommand{\Aut}{{\operatorname{Aut}}}

\newcommand{\obj}{{\operatorname{obj}}}

\newcommand{\Trace}{{\operatorname{Trace}\, }}
\newcommand{\chf}[1]{{\chi_{#1}}}
\newcommand{\Hecke}{{\mathcal{H}}}
\newcommand{\RES}[2]{{\operatorname{\textsc{res}}_\RRSch{#1}{#2} }}
\newcommand{\elliptic}{{\operatorname{ell}}}

\newcommand{\GL}{{\operatorname{GL}}}
\newcommand{\Gal}{{\operatorname{Gal}}}
\newcommand{\CHF}[3]{{\chf{\varphi_{\fais{#1},{#3}}}}}

\newcommand{\NC}[1]{{ {{\rm R}\Psi}_{#1}}}

\newcommand{\cdef}[1]{{{\it #1}}}

\begin{document}

\begin{abstract}
This paper concerns character sheaves of connected reductive algebraic groups defined over non-Archimedean local fields and their relation with characters of smooth representations.
Although character sheaves were devised with characters of representations of finite groups of Lie type in mind, character sheaves are perfectly well defined for reductive algebraic groups over any algebraically closed field.  Nevertheless, the relation between character sheaves of an algebraic group $G$ over an algebraic closure of a field $K$ and characters of representations of $G(K)$ is well understood only when $K$ is a finite field and when $K$ is the field of complex numbers.
In this paper we consider the case when $K$ is a non-Archimedean local field and explain how to match certain character sheaves of a connected reductive algebraic group $G$ with virtual representations of $G(K)$. In the final section of the paper we produce examples of character sheaves of general linear groups and matching admissible virtual representations.
\end{abstract}

%

\maketitle

\section*{Introduction}
At the beginning of the paper introducing character sheaves of connected reductive algebraic groups (\cite{CS}), George Lusztig wrote: 
\begin{quotation}
{\it  This  paper  is an attempt  to  construct  a geometric  theory  of characters  of a  reductive  algebraic  group  $G$  defined  over  an  algebraically  closed  field.  We  are seeking  a theory  which  is as close as possible  to the theory of irreducible  (complex)  characters  of the corresponding  groups  $G(\mathbb{F}_q)$  over  a finite  field  $\mathbb{F}_q$,  and  yet  it  should  have a meaning  over  algebraically  closed  fields.  The  basic  objects  in  the  theory  are  certain  irreducible  ($\ell$-adic)  perverse  sheaves \ldots on  $G$;  they  are the  analogues  of the  irreducible  ($\ell$-adic)  representations  of $G(\mathbb{F}_q)$  and  are called  the  character  sheaves of $G$.}
\end{quotation}
Using the Grothendieck-Lefschetz fixed-point formula, Lusztig went on to show that functions corresponding to Frobenius-stable character sheaves of  $G$ form a basis for the $\EE$-vector space spanned by characters of $G(\FF_q)$, in many cases. Moreover, he also introduced the machinery from which the change-of-basis matrix can be determined, justifying the appellation `character' for the perverse sheaves under consideration. 

Although character sheaves are indeed defined for connected reductive algebraic groups $\KKsch{G}$ over arbitrary algebraically closed fields $\KKq$, no relation has previously been established between character sheaves of $\KKsch{G}$ and representations of $\Ksch{G}(\Kq)$ except when $\Kq = \FF_q$ and when $\Kq = \CC$. In this paper we consider the case when $\Kq$ is a non-Archimedean local field. We introduce machinery which establishes that there is a close relation between certain character sheaves of connected reductive algebraic groups $\KKsch{G}$ defined over non-Archimedean local fields $\Kq$ and characters of certain virtual admissible representations of the group $\Ksch{G}(\Kq)$. 
%

By contrast, character sheaves of algebraic groups defined over finite fields have been used by several people (notably \cite{W}, and less notably, \cite{C}) to construct important distributions on algebraic groups  over non-Archimedean local fields  and to study characters of admissible representations. Likewise, in \cite{Ldual}, Lusztig used character sheaves of the Langlands dual group (a complex algebraic group) to classify admissible representations of $\Ksch{G}(\Kq)$, when $\Kq$ is non-Archimedean. These are quite different uses of character sheaves.

Characters of admissible representations are distributions on the Hecke algebra of $\Ksch{G}(\Kq)$. (In fact, these distributions are represented by functions on the dense subset $\Ksch{G}_{\rm reg}(\Kq)$ of $\Ksch{G}(\Kq)$ consisting of regular elements, but they do have singularities off this set.) Accordingly, in order to establish a connection between character sheaves of reductive algebraic groups defined  over non-Archimedean local fields  and characters of admissible representations, we need something like a sheaves-{\it distributions} dictionary. This paper establishes a partial result in this direction by explaining how to compare certain character sheaves of $\KKsch{G}$ with certain virtual representations of $\Ksch{G}(\Kq)$, using an idea inspired by a character formula due to Schneider-Stuhler \cite[Prop.IV.1.5]{SS}. Grossly simplified, one may say that we use the Frobenius map on the special fibres of smooth integral models corresponding to parahoric subgroups to compare the nearby cycles sheaves of character sheaves of $\KKsch{G}$ with representations of finite reductive quotients of parahoric subgroups obtained by compact restriction from virtual admissible representations of $\Ksch{G}(\Kq)$. Precisely what we mean by this is explained in the paper. We also show that certain character sheaves of $G_{\bar{\mathbb{K}}}$ naturally define distributions on $G(\mathbb{K})$.

We claim that character sheaves of reductive algebraic groups $\Ksch{G}$ defined  over non-Archimedean local fields  $\Kq$ are indeed related to characters of admissible representations by this machinery.  To support this claim, in the last section of this paper we consider the case of general linear groups. We show that if $\pi$ is a (generalised) principal series representation induced from a supercuspidal representation of a Levi subgroup of $\GL(N,\Kq)$, then there is perverse sheaf $\fais{F}$ on $\GL(N)_\KKq$ such that $(-1)^N \fais{F}$ matches $\pi$, in the sense explained in this paper.

The reason our first examples of character sheaves of algebraic groups defined  over non-Archimedean local fields are for general linear groups is that the relation between character sheaves and characters is simplest in this case. As we will explain in future work, in general, anything resembling a change-of-basis matrix between character sheaves and matching representations will involve endoscopic groups and will therefore be more complicated than the case $\GL(N)$ may suggest. However, in should be emphasized that many of the results of this paper 
apply to arbitrary connected reductive algebraic groups.

The last theorem of this paper illustrates an important aspect 
of character sheaves of algebraic groups defined over non-Archimedean local fields: it unifies parabolic and compact induction under one framework, at least for representations of $\GL(N,\Kq)$ of depth zero. This observation is the point of departure for our forthcoming  work on endoscopy and character sheaves.

In summary, this paper shows that certain aspects of the harmonic analysis of characters of admissible representations of algebraic groups over non-Archimedean local fields are encoded in character sheaves of the algebraic group itself, in certain cases at least.  It is in this sense that character sheaves of algebraic groups defined over non-Archimedean local fields are indeed sheaves for characters of representations. 

\centerline{* \ * \ * }


We now briefly state the main results and definitions of this paper. They are treated with greater precision in the body of the paper.

Let {$\Kq$} be a non-Archimedean local field and {$\Ksch{G}$} be a connected 
reductive algebraic over $\Kq$. Set $\KKsch{G} \ceq \Ksch{G} \times_\Spec{\Kq} \Spec{\KKq}$.
For each element $x$ of the (extended) Bruhat-Tits building $I(\Ksch{G},\Kq)$ of $\Ksch{G}(\Kq)$, let $\RSch{G}{x}$ be the smooth connected integral model  for $\Ksch{G}$ such that $\RSch{G}{x}(\Rq) = \Ksch{G}(\Kq)_x$. 

Suppose $\fais{F}\in \obj D^b_c(\KKsch{G},\EE)$ and $x\in I(\Ksch{G},\Kq)$. 
The special fibre $\kSch{G}{x}$ of the integral model $\RSch{G}{x}$ is an algebraic group, and we push the sheaf of nearby cycles $\NC{\RRSch{G}{x}} \fais{F}$ forward (with compact supports) from $\kkSch{G}{x}$ to the maximal reductive quotient $\nu_{\RRSch{G}{x}} : \kkSch{G}{x} \to \rkkSch{G}{x}$ of the special fibre. 
When normalized by a Tate twist in a suitable way, this sheaf is 
$$
\RESFAIS{F}{G}{x} \ceq {\nu_{\RRSch{G}{x}}}_!\ (\dim \nu_{\RRSch{G}{x}}/2)\ \NC{\RRSch{G}{x}} \fais{F},
$$
which plays the main role in this paper. (See Definition~\ref{definition: the functor}. 

After reviewing some preliminary notions in Section~\ref{section: basic notions},
in Section~\ref{section: restriction and conjugation and induction} we develop techniques (Theorems~\ref{theorem: restriction}, \ref{theorem: conjugation} and \ref{theorem: induction}) which allow us to calculate $\RESFAIS{F}{G}{x}$ as $x$ ranges over $I(\Ksch{G},\Kq)$ for many perverse sheaves $\fais{F}$ on $\KKsch{G}$. These Theorems essentially show that the functor $\RESFAIS{F}{G}{x}$ behaves nicely with respect to restriction (from one reductive quotient to another), induction (from parabolic subgroups) and also with respect to the action of $\Ksch{G}(\Kq)$ on the building.
The theorems also show that the functor $\RESFAIS{F}{G}{x}$ plays a role which is analogous to that of compact restriction of representations to finite reductive quotients of parahoric subgroups of $\Ksch{G}(\Kq)$. 

Suppose $x\in I(\Ksch{G},\Kq)$. Then $\rkkSch{G}{x}$ is defined over $\kq$ and admits a Frobenius automorphism, henceforth denoted by $\Frob_{\rkSch{G}{x}}$.
An equivariant perverse sheaf $\fais{F}$ on $\KKsch{G}$ is said to have \emph{depth zero} if, for each $x\in I(\Ksch{G},\Kq)$, there is an isomorphism 
\[
\Frob_{\rkSch{G}{x}}^* \RESFAIS{F}{G}{x} \iso \RESFAIS{F}{G}{x}
\]
in $D^b_c(\rkkSch{G}{x},\EE)$ (see Definition~\ref{definition: depth zero}).
If $\fais{F}$ is an equivariant perverse sheaf of depth zero then, for each $x\in I(\Ksch{G},\Kq)$ and for each isomorphism 
\[
\varphi_{\fais{F},x} : \Frob_{\rkSch{G}{x}}^*\  \RES{G}{x} \fais{F}\to \RES{G}{x} \fais{F},
\] 
we may use the Grothendieck-Lefschetz fixed-point formula to unambiguously define a function \cdef{$\CHF{F}{G}{x} : \rkkSch{G}{x}(\kq) \to \EE$} by
\[
\CHF{F}{G}{x}(h) \ceq \sum_{k\in \NN} (-1)^k \Trace\left((\varphi_{\fais{F},x})_h; H^k_h(
\RES{G}{x}\fais{F})\right).
\]
Thus, $\CHF{F}{G}{x}$ is the characteristic function of $\RES{G}{x}\fais{F}$ with respect to $\varphi_{\fais{F},x}$ (\cf Subsection~\ref{subsection: characteristic function}).

The class of equivariant perverse sheaves of depth zero which appear in this paper actually satisfy 
a strong property, which is encoded as compatibility relations among the isomorphisms 
$\Frob_{\rkSch{G}{x}}^* \RESFAIS{F}{G}{x} \iso \RESFAIS{F}{G}{x}$ for various $x\in I(\Ksch{G},\Kq)$.
Such compatibility relations form what we call a \emph{Frobenius structure} 
(see Definition~\ref{definition: Frobenius structure}).

Armed with these notions, we can begin to compare character sheaves of $\KKsch{G}$ with representations of $\Ksch{G}(\Kq)$. We say that an equivariant perverse sheaf $\fais{F}$ of depth zero on $\KKsch{G}$ with Frobenius structure $\varphi_\fais{F}$ \emph{matches} a virtual representation $\pi$ of $\Ksch{G}(\Kq)$ if
\[
\CHF{F}{G}{x} =  \Trace \cRes^{\Ksch{G}(\Kq)}_{\RSch{G}{x}(\Rq)}\pi,
\]
for each $x\in I(\Ksch{G},\Kq)$. In fact, under certain conditions on the group $\Ksch{G}$, every character sheaf $\fais{F}$ of $\KKsch{G}$ with Frobenius structure determines a distribution on the elliptic elements of $\Ksch{G}(\Kq)$ (explained in Subsection~\ref{subsection: SS}) and if $\fais{F}$ matches virtual representation $\pi$ then this distribution coincides with the character of $\pi$ on on the elliptic elements of $\Ksch{G}(\Kq)$ (Proposition~\ref{proposition: SS}).

Theorem~\ref{theorem: induction}, shows that if 
$\fais{F}$ is induced (in the sense explained Subsection~\ref{subsection: induction}) from a character sheaf $\fais{G}$ of an unramified maximal torus $\KKsch{T}$ defined over $\Kq$ and split over an unramified extension of $\Kq$, then we can determine $\RESFAIS{F}{G}{x}$ whenever $x\in I(\Ksch{G},\Kq)$ 
is a hyperspecial (poly)vertex in the image of $I(\Ksch{T},\Kq) \hookrightarrow I(\Ksch{G},\Kq)$. 
This result requires that we review some facts about perverse sheaves (principal fibrations, more properties of nearby cycles, and fibrations over a trait, in Subsection~\ref{subsection: fibration}), and character sheaves (parabolic induction of character sheaves and local systems, in Subsection~\ref{subsection: induction}), and also say a few words about integral models for flag varieties.

Then we consider the case $G=\GL(N)$ in order to produce examples of character sheaves and matching representations. 
The last theorem of this paper, Theorem~\ref{theorem: representations}, demonstrates how 
generalised principal series representations of $\GL(N,\Kq)$ of depth zero match with character sheaves of $\GL(N)_\KKq$ with Frobenius structure which are induced from maximal tori.
%
The key feature of  Theorem~\ref{theorem: representations} is the following. Let $\Ksch{T}$ be an unramified, maximal torus in $\GL(N)_\Kq$ and let $\fais{L}$ be a Kummer local system on $\KKsch{T}$ of depth zero. Let $\KKsch{B}\subset \KKsch{G}$ be a Borel subgroup with Levi component $\KKsch{T}$. Cohomological parabolic induction produces a perverse sheaf $\ind^{\KKsch{G}}_{\KKsch{B}}\fais{L}[N]$ which is a finite direct sum of character sheaves of $\KKsch{G}$. 
This is true regardless of whether or not $\Ksch{T}$ is split over $\Kq$! If $\Ksch{T}$ is elliptic over $\Kq$ then $\ind^{\KKsch{G}}_{\KKsch{B}}\fais{L}[N]$ matches a supercuspidal representation; otherwise, the perverse sheaf $\ind^{\KKsch{G}}_{\KKsch{B}}\fais{L}[N]$ matches a generalised principal series representation of $\Ksch{G}(\Kq)$. 

This property of character sheaves is the point of departure for forthcoming work on endoscopy and character sheaves.
\medskip

\centerline{* \ * \ * }

This paper could not have been written without the wealth of mathematics the first author learned over the years while working with Anne-Marie Aubert and he is happy to have this opportunity to acknowledge her role in this work. 
He also thanks Jiu-Kang Yu for laying much of the ground work for this paper with his work on smooth integral models for $p$-adic groups.
Finally, he thanks the Institut des Hautes \'Etudes Scientifique for wonderful hospitality while some of the ideas for this paper were developed, and Laurent Fargues and Christophe Breuil for helpful conversations there.
An early version of this material was presented at the Quebec-Vermont Number Theory Seminar at McGill University in March 2007; the main results in this paper was presented at the Dieuxi\`eme Congress Canada-France in the Automorphic Forms Session organised by Stephen Kudla and Colette Moeglin in June 2008.

\tableofcontents


\section{Preliminaries}\label{section: basic notions}

In this section we review some important properties of perverse sheaves and character sheaves and also set notation for the rest of the paper. This section may be skipped if the reader is willing to refer back when necessary.

Let \cdef{$\sch{S} = \{ {\eta}, {\bf s} \}$} be a Henselian trait with closed point ${\bf s}$, generic point $\eta$ and generic geometric point \cdef{${\bar \eta}$} defining \cdef{$\bar{\bf s}$} over ${\bf s}$. Let \cdef{$\sch{\bar S} = \{ {\bar \eta}, {\bar{\bf s}} \}$}. 
Then $\mathcal{O}_{\sch{\bar S}}({\bar \eta})$ is an algebraic closure of a non-Archimedean local field $\mathcal{O}_{\sch{S}}({\eta})$, and $\mathcal{O}_{\sch{\bar S}}({\bar{\bf s}})$ is an algebraic closure of the residue field $\mathcal{O}_{\sch{S}}({\bf s})$ (a finite field) of $\mathcal{O}_{\sch{S}}({\eta})$. It should be noted that we have not placed any restriction on the characteristic of $\mathcal{O}_{\sch{S}}({\eta})$. 

Let \cdef{$\ell$} be a prime invertible in $\mathcal{O}_{\sch{S}}({\bf s})$ (and therefore in $\mathcal{O}_{\sch{S}}(\sch{S})$) and henceforth fixed.

In this paper, we generally work with
schemes of finite type over $\sch{S}$ or $\sch{\bar S}$ or an algebraic varieties 
over ${\eta}$ or ${\bf s}$ or ${\bar \eta}$ or ${\bar{\bf s}}$. In each case we will make use of the `derived' category \cdef{$D^b_c(\sch{X},\EE)$} of cohomologically bounded constructible $\ell$-adic sheaf complexes, where $\ell$ is a prime number different from the characteristic of $\mathcal{O}_{\sch{S}}({s})$. This category was introduced in \cite[1.1.1-1.1.5]{D2}; see also \cite[2.2.9, 2.2.14, 2.2.18]{BBD} (\cf \cite[expos\'es VI, V, XV]{SGA5}).\footnote{In fact, the definition found there does not in general define a triangulated category, as mentioned in \cite[0.6]{Lau}; this problem is resolved in \cite{Ek}, in which it is shown that $D^b_c(\sch{X},\EE)$ is a triangulated category satisfying the yoga of the `six functors closest to Grothendieck's heart' and also admitting Hom and tensor products.}
As much as possible, we follow the notational conventions of \cite{BBD} (and therefore \cite{CS} and \cite{MS}) regarding derived functors.

\subsection{Base Change}\label{subsection: base change}

In this paper we use several versions of what is collectively known as base change. In this subsection we state these results. In each case, the proof is based on results from \cite{SGA4} (torsion sheaves) adapted to the category $D^b_c(\sch{X},\EE)$ by the adic formalism, such as found in \cite{Ek}.\footnote{In fact, this `adaptation' requires considerable effort, which we do not include here. In this regard we follow standard practice: ``On utilisera librement pour les $\EE$-faisceaux et leur cat\'egorie d\'eriv\'ee ... les th\'eor\`emes qui sont \'enonc\'es et \'etablis dans la litt\'erature pour les faisceaux constructibles de $\Lambda$-modules, $\Lambda$ un anneau fini d'ordre premier \`a $p$ ...'' \cite[0.5]{Lau}.}

Consider the following Cartesian diagram in the category of schemes:
\begin{equation}\label{diagram: Cartesian}
	\xymatrix{
	\sch{X}' \ar[r]^{\sch{g}'} \ar[d]_{\sch{f}'} & \sch{X} \ar[d]^{\sch{f}} \\
	\sch{Y}' \ar[r]_{\sch{g}} & \sch{Y}.
	}
\end{equation}
Let
\begin{equation}\label{equation: base change}
\xymatrix{
\sch{g}^*\ \sch{f}_* \ar[r] & \sch{f}'_*\ {\sch{g}'}^* 
}
\end{equation}
be the base change morphism, as a morphism of functors $D^b_c(\sch{X},\EE) \to D^b_c(\sch{Y}',\EE)$ (\cf \cite[XII, \S 4]{SGA4}).

\subsubsection{Proper Base Change}\label{item: PBC}
If $f$ is proper, then \eqref{equation: base change} is an isomorphism of functors. 
This follows from \cite[XII, 5.1(i)]{SGA4} and the adic formalism.

\subsubsection{Smooth Base Change}\label{item: SBC}
If $g$ is smooth, then \eqref{equation: base change} is an isomorphism of functors.
This follows from \cite[XVI]{SGA4} and the adic formalism.


We will also need base change as it pertains to compact supports (\cf \cite[XVII, 5]{SGA4}).
Referring again to Diagram~\ref{diagram: Cartesian}, consider the base change morphism,
\begin{equation}\label{equation: base change with compact supports}
\xymatrix{
\sch{g}^*\ \sch{f}_! \ar[r] & \sch{f}'_!\ {\sch{g}'}^* 
}
\end{equation}
as a morphism of functors $D^b_c(\sch{X},\EE) \to D^b_c(\sch{Y}',\EE)$.

\subsubsection{Proper Base Change with Compact Supports}\label{item: FBC}
If $\sch{f}$ is proper, then $\sch{f}_! = \sch{f}_*$, so \eqref{equation: base change with compact supports} is an isomorphism of functors by Subsection~\ref{item: PBC}.

\subsubsection{Base Change with Compact Supports}\label{item: BCC}

If $\sch{f}$ is locally of finite type and separated,
then \eqref{equation: base change with compact supports} is an isomorphism of functors.
This follows from \cite[XVII, Prop.6.1.4~(iii)]{SGA4} and the adic formalism.
(A closely related version of base change with compact supports appears in \cite[(1.7.5)]{CS}.)

\subsection{Nearby Cycles}\label{section: nearby}
 
Let \cdef{$\Sch{X}$} be a scheme of finite type over $\sch{S}$ and set \cdef{$\Sch{\bar X} = \Sch{X}\times_\sch{S} \sch{\bar S}$}. 
Consider the diagram
\[
\xymatrix{
\ar[d] \Sch{X}_{\bar \eta} \ar[r]^{j_{\Sch{\bar X}}} & \ar[d] \Sch{\bar X} & \ar[d] \ar[l]_{i_{\Sch{\bar X}}} \Sch{X}_{\bar{\bf s}} \\
{\bar \eta} \ar[r] & \bar{S} & \ar[l] {\bar{\bf s}}, 
}
\]
where \cdef{$j_\Sch{\bar X}$} and \cdef{$i_\Sch{\bar X}$} are obtained by pull-back (so the squares are Cartesian).
The {\it nearby cycles functor for $\Sch{\bar X}$} is defined by
\[
\NC{\Sch{\bar X}}\ceq {i_{\Sch{\bar X}}}^*\ {j_{\Sch{\bar X}}}_*.
\]
Recall that, as much as possible, we follow \cite{BBD} regarding notation for derived functors (although the definitive reference for nearby cycles is \cite{SGA7}, of course). 

We now record some basic properties of the nearby cycles functor which will be important later.

Since $\Sch{X}\to \sch{S}$ is of finite type, the functor $\NC{\Sch{\bar X}}$ preserves constructibility (\cf \cite[Th. finitude]{SGA4.5} and \cite[\S1.1]{Ill}). Thus, we may write
\begin{equation}\label{equation: NC bc}
\NC{\Sch{\bar X}} : D^b_c(\Sch{X}_{\bar \eta} , \EE) \to D^b_c(\Sch{X}_{\bar{\bf s}}, \EE).
\end{equation}

Next, suppose $\Sch{Y}\to S$ is also of finite type and let $\Sch{h} : \Sch{X}\to \Sch{Y}$ be a morphism over $\sch{S}$.  Consider the diagram
\[
		\xymatrix{
\Sch{X}_{\bar \eta} \ar[d]^{\Sch{h}_{\bar \eta}} \ar[r]^{j_{\Sch{\bar X}}} & \Sch{\bar X} \ar[d]^{\Sch{\bar h}} & \ar[l]_{i_{\Sch{\bar X}}} \ar[d]^{\Sch{h}_{\bar{\bf s}}} \Sch{X}_{\bar{\bf s}} \\
\Sch{Y}_{\bar \eta} \ar[r]^{j_{\Sch{\bar Y}}} & \Sch{\bar Y} & \ar[l]_{i_{\Sch{\bar Y}}} \Sch{Y}_{\bar{\bf s}} 
		}
\]
Again, each square is Cartesian.

\subsubsection{Proper Base Change and Nearby Cycles}\label{item: PNC_*}

If $\Sch{h} : \Sch{X} \to \Sch{Y}$ is proper then 
it follows (\cf \cite[XIII, 1.3.6]{SGA7})  from proper base change (Subsection~\ref{item: PBC}) that there is a canonical isomorphism of functors 
$\NC{\Sch{\bar Y}} \ {\Sch{h}_{\bar \eta}}_*  \to {\Sch{h}_{\bar{\bf s}}}_*\ \NC{\Sch{\bar X}}$.

\subsubsection{Smooth Base Change and Nearby Cycles}\label{item: SNC^*}

If $\Sch{h} : \Sch{X} \to \Sch{Y}$ is smooth then 
it follows (\cf  \cite[XIII, 1.3.7]{SGA7}) from smooth base change (Subsection~\ref{item: SBC}) that there is a canonical isomorphism of functors 
${\Sch{h}_{\bar{\bf s}}}^*\ \NC{\Sch{\bar Y}}  \to \NC{\Sch{\bar X}} \ {\Sch{h}_{\bar \eta}}^*$.

\subsubsection{Compact Supports and Nearby Cycles}\label{item: FNC_!}

If $\Sch{h} : \Sch{X} \to \Sch{Y}$ is proper then it follows (\cf  \cite[XIII, 1.3.8]{SGA7}) from proper base change (Subsection~\ref{item: FBC}) that there is a canonical isomorphism of functors 
$\NC{\Sch{\bar Y}} \ {\Sch{h}_{\bar \eta}}_!  \to {\Sch{h}_{\bar{\bf s}}}_!\ \NC{\Sch{\bar X}}$.
In fact, this is easy to see directly from proper base change: if $\Sch{h} : \Sch{X}\to \Sch{Y}$ is proper then $\Sch{h}_!= \Sch{h}_*$, in which case $\NC{\Sch{\bar Y}} \ {\Sch{h}_{\bar \eta}}_!  \iso {\Sch{h}_{\bar{\bf s}}}_!\ \NC{\Sch{\bar X}}$ by proper base change (\cf Subsection~\ref{item: PNC_*}).

\subsection{Perverse Sheaves}

\subsubsection{Equivariant Perverse Sheaves}\label{subsection: equivariant perverse sheaves}

In this subsection we work in the category of algebraic varieties over $\mathcal{O}_{\sch{\bar S}}({\bar \eta})$ or $\mathcal{O}_{\sch{\bar S}}({\bar{\bf s}})$. Following \cite{BBD}, we denote the category of perverse sheaves on $\sch{X}$ by \cdef{$\catM\sch{X}$}.

Let \cdef{$\sch{m} : \sch{H}\times \sch{X} \to \sch{X}$} be an action of a {connected} algebraic group \cdef{$\sch{H}$} on $\sch{X}$. Recall from \cite[\S0]{L0} that $\fais{F}\in \catM\sch{X}$ is an \cdef{equivariant perverse sheaf on $\sch{X}$} if there is an isomorphism 
	\begin{equation}\label{equation: equivariant}
	\cdef{\mu_{\fais{F}} : \sch{m}^* \fais{F} \to \proj^* \fais{F}}
	\end{equation}
in $D^b_c(\sch{H}\times\sch{X}, \EE)$ such that $\sch{e}^*\mu_\fais{F} = \id_\fais{F}$, where  \cdef{$\sch{e} : \sch{X} \to \sch{H}\times\sch{X}$} is defined by $x \mapsto (1,x)$ and \cdef{$\proj : \sch{H}\times \sch{X} \to \sch{X}$} is projection onto the second component. As observed in \cite[\S0]{L0}, if $\fais{F}$ is an equivariant perverse sheaf, then $\mu_\fais{F}$ is essentially unique. 

A morphism $\phi : \fais{F}_1\to \fais{F}_2$ perverse sheaves on $\sch{X}$ is an \cdef{equivariant morphism of perverse sheaves} if the following diagram commutes.
\[
\xymatrix{
	\ar[d]_{\mu_{\fais{F}_1}} \sch{m}^*\fais{F}_1 \ar[r]^{\sch{m}^*\phi} & \ar[d]^{\mu_{\fais{F}_2}} \sch{m}^*\fais{F}_2\\
	\proj^*\fais{F}_1 \ar[r]^{\proj^*\phi} & \proj^*\fais{F}
}
\]
Note that this definition makes implicit use of the essential uniqueness of the isomorphisms $\sch{m}^* \fais{F}_1 \to \proj^* \fais{F}_1$ and $\sch{m}^* \fais{F}_2 \to \proj^* \fais{F}_2$ as above. Since $\id_\fais{F}$ is equivariant if $\fais{F}$ is equivariant and since the composition of equivariant morphisms is equivariant, it follows that $\sch{H}$-equivariant perverse sheaves on $X$ form a category, with morphisms as above; this category is denoted by \cdef{$\catM_\sch{H}\sch{X}$}.

We finish this subsection with a comment that will be used in Subsection~\ref{subsection: conjugation}.
For each $h\in \sch{H}$, let $\sch{e}_h : \sch{X} \to \sch{H} \times \sch{X}$ be the morphism determined by $\sch{e}_h(x) = (h,x)$; note that $\proj \circ \sch{e}_{h^{-1}} = \id$. Define
	\begin{equation}\label{equation: equivariance}
	\cdef{\mu_{\fais{F}}(h) = {\sch{e}_{h^{-1}}}^*\ \mu_{\fais{F}}}.
	\end{equation}
Define \cdef{$\sch{m}(h^{-1}) : \sch{H}\to \sch{H}$} by $\sch{m}(h^{-1}) \ceq \sch{m}\circ \sch{e}_{h^{-1}}$. Then \eqref{equation: equivariance} defines a family of isomorphisms
	\begin{equation}\label{equation: iii reductive}
	\forall h\in \sch{H},\qquad \mu_{\fais{F}}(h): \sch{m}(h^{-1})^*\ \fais{F} \to \fais{F}.
	\end{equation}

\subsubsection{Parabolic Restriction}\label{subsection: restriction}

In this subsection we work in the category of algebraic varieties over $\mathcal{O}_{\sch{\bar S}}({\bar \eta})$ or $\mathcal{O}_{\sch{\bar S}}({\bar{\bf s}})$.

Let \cdef{$\iota : \sch{P}\to \sch{G}$} be a parabolic subgroup and let \cdef{$\pi_P : \sch{P}\to \sch{L}$} be the quotient map to the Levi component of $\sch{P}$. 
\begin{equation}
\xymatrix{
\sch{G} & \ar[l]_{\iota} \sch{P} \ar[r]^{\pi_{P}} & \sch{L}
}
\end{equation}
The functor \cdef{$\res^\sch{G}_\sch{P} : D^b_c(\sch{G},\EE) \to D^b_c(\sch{L},\EE)$} is defined by 
\begin{equation}\label{equation: restriction}
\res^\sch{G}_\sch{P} \ceq  {\pi_P}_!\ {\iota}^\ast
\end{equation}
when working in the category of algebraic varieties over $\mathcal{O}_{\sch{\bar S}}({\bar \eta})$, and by
\begin{equation}\label{equation: restriction finite}
\res^\sch{G}_\sch{P} \ceq  {\pi_P}_!\, (\dim\pi_P)\ {\iota}^\ast
\end{equation}
when working in the category of algebraic varieties over $\mathcal{O}_{\sch{\bar S}}({\bar{\bf s}})$, where $ (\dim\pi_P)$ indicates Tate twist by  $\dim\pi_P$.




\subsubsection{Principal Fibrations}\label{subsection: fibration}

In this subsection we review a fundamental result concerning equivariant perverse sheaves. 
This result, Proposition~\ref{proposition: fibration}, is stated in \cite[1.9.3]{CS}, \cite[p.65]{BBD} and in \cite[1.4.2]{MS}; in each case a proof is sketched. 

In this subsection we work in the category of algebraic varieties over $\mathcal{O}_{\sch{\bar S}}({\bar \eta})$ or $\mathcal{O}_{\sch{\bar S}}({\bar{\bf s}})$ (where $S$ is a Henselian trait, see Section~\ref{section: basic notions}).

\begin{proposition}[Beilinson-Bernstein-Deligne]\label{proposition: fibration}
Let $\sch{f} : \sch{X} \to \sch{Y}$ be a principal fibration with group $\sch{H}$. Suppose $\sch{H}$ is connected. 
If $\fais{F}$ is a perverse sheaf on $\sch{X}$, then $\fais{F}$ is $\sch{H}$-equivariant if and only if $\fais{F} \iso \sch{f}^*[\dim\sch{H}] \mathcal{G}$ for some perverse sheaf $\mathcal{G}$ on $\sch{Y}$.
\end{proposition}


\begin{corollary}\label{corollary: fibration}
Let $\sch{f} : \sch{X} \to \sch{Y}$ be a principal fibration with group $\sch{H}$. Suppose $\sch{H}$ is connected. Then $\sch{f}^*[\dim\sch{H}] : \catM\sch{Y} \to \catM_\sch{H}\sch{X}$ is an equivalence of categories and $\catM_\sch{H}\sch{X}$ is a thick subcategory of $\catM\sch{X}$.
\end{corollary}

\begin{proof}
By \cite[Prop~4.2.5]{BBD} we know that $\sch{f}^*[\dim\sch{H}] : \catM\sch{Y} \to \catM\sch{X}$ is full and faithful. The proof of Proposition~\ref{proposition: fibration} shows that $\sch{f}^*[\dim\sch{H}] \mathcal{G}$ is an equivariant perverse sheaf on $\sch{X}$ for each perverse sheaf $\mathcal{G}$ on $\sch{Y}$ and that $\sch{f}^*[\dim\sch{H}] \phi$ is an equivariant morphism in $\catM\sch{X}$ for each morphism $\phi$ in $\catM\sch{Y}$. Thus, $\sch{f}^*[\dim\sch{H}]$ is a full and 
faithful functor from $\catM\sch{Y}$ to $\catM_\sch{H}\sch{X}$. 
Proposition~\ref{proposition: fibration} tells us that this functor is essentially surjective. Thus, $\sch{f}^*[\dim\sch{H}]$ is an equivalence. The last clause of Corollary~\ref{corollary: fibration} follows from \cite[4.2.6]{BBD}.
\end{proof}

Let $\sch{f} : \sch{X} \to \sch{Y}$ be a principal fibration with group $\sch{H}$. Suppose $\sch{H}$ is connected. We write 
\begin{equation}
\cdef{\sch{f}_\# : \catM_\sch{H}\sch{X} \to \catM\sch{Y}}
\end{equation}
for the (essentially unique) inverse of the equivalence $\sch{f}^*[\dim\sch{H}] : \catM\sch{Y} \to \catM_\sch{H}\sch{X}$ (\cf Corollary~\ref{corollary: fibration}).

%
%

\subsubsection{Nearby Cycles and Principal Fibrations}

Let \cdef{$\Sch{X}$}, \cdef{$\Sch{Y}$} and \cdef{$\Sch{H}$} be schemes of finite type over the Henselian trait $\sch{S}$. 

With a bit of work, it follows from \cite[\S4.4]{BBD} that $\NC{\Sch{\bar X}}$ takes perverse sheaves on $\Sch{X}_{\bar \eta}$ to perverse sheaves on $\Sch{X}_{\bar{\bf s}}$.
Moreover, if $\Sch{H}\times \Sch{X} \to \Sch{X}$ is a smooth action of a connected group over $\sch{S}$, then $\NC{\Sch{\bar X}}$ takes $\Sch{H}_{\bar \eta}$-equivariant perverse sheaves on $\Sch{X}_{\bar \eta}$ to $\Sch{H}_{\bar{\bf s}}$-equivariant perverse sheaves on $\Sch{X}_{\bar{\bf s}}$; thus, we may write
\begin{equation}\label{equation: EPS}
	\NC{\Sch{\bar X}} : \catM_{\Sch{H}_{\bar \eta}}\Sch{X}_{\bar \eta} \to \catM_{\Sch{H}_{\bar{\bf s}}}\Sch{X}_{\bar{\bf s}}.
\end{equation}
(Compare with \eqref{equation: NC bc}.)

\subsubsection{Fibrations over a Trait}\label{subsection: NCF}

As above, let $\Sch{X}$, $\Sch{Y}$ and $\Sch{H}$ be schemes of finite type over the Henselian trait $\sch{S}$. The following proposition is used in the proof of Theorem \ref{theorem: induction}.

\begin{proposition}\label{proposition: NCe}
Let $\Sch{a} : \Sch{X}\to \Sch{Y}$ be a smooth morphism of schemes over $\sch{S}$ 
such that its generic and special fibres are smooth principal fibrations with 
connected groups
$\Sch{H}_{\bar \eta}$ and $\dim\Sch{H}_{\bar{\bf s}}$ such that
$\dim \Sch{H}_{\bar \eta}=\dim\Sch{H}_{\bar{\bf s}}$. Then
\[
\NC{\Sch{\bar Y}}\ (\Sch{a}_{\bar \eta})_\# \ \fais{F} \iso (\Sch{a}_{\bar{\bf s}})_\#\ \NC{\Sch{\bar X}}\ \fais{F}
\]
for all $\fais{F} \in \obj \catM_{\Sch{H}_{\bar \eta}} \Sch{X}_{\bar \eta}$.
\end{proposition}

\begin{proof}
From the assumptions it follows that the maps $(\Sch{a}_{\bar \eta})_\# : \catM_{\Sch{H}_{\bar \eta}} \Sch{X}_{\bar \eta} \to \catM\Sch{Y}_{\bar \eta}$ and $(\Sch{a}_{\bar{\bf s}})_\# : \catM_{\Sch{H}_{\bar{\bf s}}} \Sch{X}_{\bar{\bf s}} \to \catM\Sch{Y}_{\bar{\bf s}}$ are well-defined (\cf Subsection~\ref{subsection: fibration}).  
Suppose $\fais{F} \in \obj \catM_{\Sch{H}_{\bar \eta}} \Sch{X}_{\bar \eta}$ and let $\fais{G} = (\Sch{a}_{\bar \eta})_\# \ \fais{F}$.  So $\fais{G}\in \obj\catM\Sch{Y}_{\bar \eta}$ and $\fais{F} = {\Sch{a}_{\bar \eta}}^*\ [\dim\Sch{H}_{\bar \eta}]$ (\cf Subsection~\ref{subsection: fibration}). Since $\Sch{a}$ is smooth by hypothesis, Subsection~\ref{item: SNC^*} provides an isomorphism of sheaves
\[
{\Sch{a}_{\bar{\bf s}}}^*\ \NC{\Sch{\bar Y}} \fais{G} \iso  \NC{\Sch{\bar X}}\ {\Sch{a}_{\bar \eta}}^*\ \fais{G}.
\]
Since $\dim \Sch{H}_{\bar \eta} = \dim \Sch{H}_{\bar{\bf s}}$, it follows that
\[
{\Sch{a}_{\bar{\bf s}}}^*\ [\Sch{H}_{\bar{\bf s}}]\ \NC{\Sch{\bar Y}} \fais{G} \iso  \NC{\Sch{\bar X}}\ {\Sch{a}_{\bar \eta}}^*\ [\Sch{H}_{\bar \eta}] \fais{G}.
\]
Since the inverse of ${\Sch{a}_{\bar{\bf s}}}^*\ [\Sch{H}_{\bar{\bf s}}]$ is $(\Sch{a}_{\bar{\bf s}})_\#$ (as defined in Subsection~\ref{subsection: fibration}), we have
\[
\NC{\Sch{\bar Y}}  \fais{G} \iso  (\Sch{a}_{\bar{\bf s}})_\#\ \NC{\Sch{\bar X}}\ {\Sch{a}_{\bar \eta}}^*\ [\Sch{H}_{\bar \eta}] \fais{G};
\]
likewise, since the inverse of ${\Sch{a}_{\bar{\eta}}}^*\ [\Sch{H}_{\bar{\eta}}]$ is $(\Sch{a}_{\bar{\eta}})_\#$,
\[
\NC{\Sch{\bar Y}} \ (\Sch{a}_{\bar \eta})_\# \ \fais{F} \iso  (\Sch{a}_{\bar{\bf s}})_\#\ \NC{\Sch{\bar X}}\fais{F},
\]
as desired.
\end{proof}

\subsection{Character Sheaves}

\subsubsection{Parabolic Induction of Character Sheaves}\label{subsection: induction}

In this subsection we review the definition of parabolic induction of character sheaves introduced in \cite[\S4]{CS} (see also \cite[\S7.1.1]{MS})
Notation here is the same as employed in Subsection~\ref{subsection: restriction}. In particular, for the moment we work in the category of algebraic varieties over $\mathcal{O}_{\sch{\bar S}}({\bar \eta})$ or $\mathcal{O}_{\sch{\bar S}}({\bar{\bf s}})$.

Consider the varieties
\begin{equation}\label{equation: induction varieties}
\begin{aligned}
		&\cdef{\sch{X}_\sch{P}\ceq \left\{ (g,h)\in \sch{G}\times \sch{G} \tq h^{-1}g h \in \sch{P}\right\}} \\
		&\cdef{\sch{Y}_\sch{P}\ceq \left\{ (g,h\sch{P})\in \sch{G}\times (\sch{G}/\sch{P}) \tq h^{-1}g h \in \sch{P}\right\}}
\end{aligned}
\end{equation}
and the diagram
	\begin{equation}
		\xymatrix{
		\sch{L} & \ar[l]_{\alpha_\sch{P}} \sch{X}_\sch{P} \ar[r]^{\beta_\sch{P}} & \sch{Y}_\sch{P} \ar[r]^{\gamma_\sch{P}} & \sch{G}
		}
	\end{equation}
in which $\cdef{\alpha_\sch{P}: \sch{X}_\sch{P} \to \sch{L}}$ is defined by $\alpha_\sch{P}(g,h) \ceq \pi_P(h^{-1}g h)$, $\cdef{\beta_\sch{P} : \sch{X}_\sch{P} \to \sch{Y}_\sch{P}}$ is defined by $\beta_\sch{P}(g,h) \ceq (g,h\sch{P})$ and $\cdef{\gamma_\sch{P} : \sch{Y}_\sch{P} \to \sch{G}}$ is defined by $\gamma_\sch{P}(g,h\sch{P}) \ceq g$. Then $\alpha_\sch{P}$ is $\sch{P}$-equivariant for the action of $\sch{P}$ on $\sch{L}$ given by $p\cdot l \mapsto \pi_P(p) l \pi_P(p)^{-1}$,
and the action of $\sch{P}$ on $\sch{X}_\sch{P}$ given by $p\cdot (g,h) \mapsto (g,hp^{-1})$. Likewise, $\beta_\sch{P} : \sch{X}_\sch{P}\to \sch{Y}_\sch{P}$ is $\sch{P}$-equivariant for the action of $\sch{P}$ on $\sch{X}_\sch{P}$ just defined and the trivial action of $\sch{P}$ on $\sch{Y}_\sch{P}$. We remark that $\gamma_\sch{P} : \sch{Y}_\sch{P} \to \sch{G}$ is proper.

Let $\fais{G}$ be an equivariant perverse sheaf on $\sch{L}$. Then Lusztig's parabolic induction (see \cite[\S4]{CS} and \cite[\S7.1.1]{MS}) is given by
\begin{equation}\label{equation: induction functor}
\cdef{\ind^\sch{G}_\sch{P} \fais{G} \ceq {\gamma_\sch{P}}_!\ (\beta_\sch{P})_\# \ {\alpha_\sch{P}}^* \ \fais{G} [\dim \sch{G} + \dim\pi_P]},
\end{equation}
where $(\beta_\sch{P})_\#$ is defined in Subsection~\ref{subsection: fibration}.
(Since $\gamma_\sch{P}$ is proper, ${\gamma_\sch{P}}_! = {\gamma_\sch{P}}_*$.)
To see that the definition of $\ind^\sch{G}_\sch{P} \fais{G}$ makes sense, observe that ${{\alpha}_\sch{P}}^*\ \fais{G} [\dim \sch{G} + \dim\pi_P]$ is a $\sch{P}$-equivariant perverse sheaf on $\sch{X}_\sch{P}$ by Proposition~\ref{proposition: fibration} and that $\beta_\sch{P}: \sch{X}_\sch{P} \to \sch{Y}_\sch{P}$ is a principal fibration with group $\sch{P}$.

If $\fais{G}$ is a strongly cuspidal character sheaf of $\sch{L}$ then $\ind^\sch{G}_\sch{P} \fais{G}$ is a semisimple $\sch{G}$-equivariant perverse sheaf on $\sch{G}$ and its irreducible summands are character sheaves of $\sch{G}$. Every character sheaf of $\sch{G}$ occurs in this way \cite[Thm.~9.3.2]{MS}.

\subsubsection{Local Systems}\label{subsection: local systems}

For use in Subsection~\ref{subsection: unramified tori}, we now recall a few basic facts about Kummer local systems on algebraic tori, following \cite[\S\S 1.11, 2.1, 2.2]{CS} for the most part. 

For each integer $d$, let \cdef{$\sch{[d]} : \sch{\GL(1)}\to \sch{\GL(1)}$} be the morphism of schemes defined by $t\mapsto t^{d}$ on a global coordinate $t$ for $\sch{\GL(1)}$. If $d$ is non-zero, we may consider the Kummer sequence (of schemes) below. 
\[
\xymatrix{
\sch{\mu}_{d} \ar[r] & \sch{\GL(1)} \ar[r]^{\sch{[d]}} & \sch{\GL(1)}.
}
\]

Fix an injective character \cdef{$\psi : \mu(\bar\eta) \to \EE^\times$} of the group of roots of unity in $\mathcal{O}(\bar\eta)$. 
If $d$ is invertible in $\mathcal{O}(\bar\eta)$, then $[d]_{\bar\eta} : \GL(1)_{\bar\eta} \to \GL(1)_{\bar\eta}$ is a Galois cover with group $\mu_{d,\bar\eta}$. Now, $\mu_{d,\bar\eta}$ acts on the local system ${\sch{[d]}_{\bar\eta}}_*\, (\EE)_{\sch{\GL(1)}_{\bar\eta}}$ (where $(\EE)_{\sch{\GL(1)}_{\bar\eta}}$ is the constant sheaf on $\sch{\GL(1)}_{\bar\eta}$). The summand of ${\sch{[d]}_{\bar\eta}}_*\, (\EE)_{\sch{\GL(1)}_{\bar\eta}}$ on which $\mu_{d,\bar\eta}$ acts according to the character $\psi$ is a rank-one local system on $\sch{\GL(1)}_{\bar\eta}$, denoted by \cdef{$\fais{E}_{d,\psi}$} \cite[1.12]{CS}. 

For each character $\sch{[n]}_{\bar\eta}$ of $\sch{\GL(1)}_{\bar\eta}$, the $\ell$-adic sheaf ${\sch{[n]}_{\bar\eta}}^*\ \fais{E}_{d,\psi}$ is also a rank-one local system on $\sch{\GL(1)}_{\bar\eta}$. The rule $(n,\frac{1}{d}) \mapsto {\sch{[n]}_{\bar\eta}}^*\, \fais{E}_{d,\psi}$ (with $n$ any integer, $d$ invertible in $\mathcal{O}(\eta)$, and $\psi$ fixed, as above) defines a group homomorphism from $\ZZ_{(p')}$ (the localisation of the ring $\ZZ$ at the prime ideal $(p')$) onto the group (with respect to tensor products) \cdef{$\mathcal{K}\sch{\GL(1)}_{\bar\eta}$} of Kummer local systems on $\sch{\GL(1)}_{\bar\eta}$, and an isomorphism $\ZZ_{(p')}/\ZZ \to \mathcal{K}\sch{\GL(1)}_{\bar\eta}$, where $p'$ is the characteristic of $\mathcal{O}(\eta)$ (so $p'$ is equal to the characteristic of $\mathcal{O}({\bf s})$, or $p'=0$).

More generally, if $\sch{T}_{\bar\eta}$ is an algebraic torus over $\bar\eta$
then 
\begin{eqnarray*}
X(\sch{T}_{\bar\eta})\otimes_\ZZ \ZZ_{(p')}/\ZZ &\to \mathcal{K}\sch{T}_{\bar\eta} \\
(\lambda, \frac{1}{d}) &\mapsto \lambda^*\, \fais{E}_{d,\psi}
\end{eqnarray*}
is an isomorphism of groups (with addition on the left-hand side, tensor product on the right-hand side), where \cdef{$X(\sch{T}_{\bar\eta})$} is the character lattice for $\sch{T}_{\bar\eta}$, where \cdef{$\mathcal{K}\sch{T}_{\bar\eta}$} is the class of Kummer local systems on $\sch{T}_{\bar\eta}$ and where \cdef{$p'$} is the characteristic of $\mathcal{O}(\eta)$.

Likewise, if we fix an injective character \cdef{$\bar\psi : \mu(\bar{\bf s}) \to \EE^\times$}  of the group of roots of unity in $\mathcal{O}(\bar{\bf s})$, and if $\sch{T}_{\bar{\bf s}}$ is an algebraic torus over $\bar{\bf s}$, then $(\lambda, \frac{1}{d}) \mapsto \lambda^*\, \fais{E}_{d,\bar\psi}$
defines an isomorphism $X(\sch{T}_{\bar{\bf s}})\otimes_\ZZ \ZZ_{(p)}/\ZZ \to \mathcal{K}\sch{T}_{\bar{\bf s}}$, where $X(\sch{T}_{\bar{\bf s}})$ is the character lattice for $\sch{T}_{\bar{\bf s}}$ and where $p$ is the characteristic of $\mathcal{O}({\bf s})$.


Finally, we recall that character sheaves of algebraic tori are simply Kummer local systems as sheaf complexes concentrated at the dimension of the tori, as explained in \cite[\S 2.10]{CS}.

\subsection{Characteristic Functions}\label{subsection: characteristic function}

Recall that we fixed a Henselian trait $\sch{S}$ in Section~\ref{section: basic notions}. In order to simplify notation slightly, we will now write  \cdef{$\kq$} for the field $\mathcal{O}_{\sch{S}}({\bf s})$ (a finite field) and $q$ for the cardinality of $\kq$; we will also write \cdef{$\kkq$} for the algebraically closed field $\mathcal{O}_{\sch{S}}({\bar{\bf s}})$.

Let \cdef{$\sch{X}$} be a $\kkq$-variety defined over $\kq$; let $\Frob_\sch{X}$ be a Frobenius for $\sch{X}$. An object $\fais{A}$ from $D^b_c(\sch{X},\EE)$ is said to be \cdef{Frobenius-stable} if there is an isomorphism \cdef{$\varphi_\fais{A} : \Frob_\sch{X}^* \fais{A} \to \fais{A}$} in $D^b_c(\sch{X},\EE)$.  If this is the case then, following \cite[8.4]{CS}, the Grothendieck-Lefschetz fixed point formula defines a function \cdef{$\chf{\varphi_\fais{A}} : \sch{X}(\kq) \to \EE$}, called the \cdef{characteristic function} of $\fais{A}$, according to the expression
\[
\chf{\varphi_\fais{A}}(a) \ceq \sum_{k\in \NN} (-1)^k  (-1)^k \Trace\left( (\varphi_{\fais{A}})_a; \mathcal{H}^k_a(\fais{A}) \right),
\]
for all $a\in \sch{X}(\kq)$, where $\mathcal{H}^k_a(\fais{A})$ is the stalk at $h$ of the $k$-th cohomology sheaf of $\fais{A}$ (\cf \cite[Eqn.8.4.1]{CS} and compare with \cite[\S 1.1.1]{Lau}).
If the isomorphism $\varphi_\fais{A}$ is understood, then we may write \cdef{$\chf{\fais{A}}$} for $\chf{\varphi_\fais{A}}$ (as we do, for example, in the proof of Theorem~\ref{theorem: representations}).

In this section we record some properties of Frobenius-stable sheaves without proof. We begin by reviewing some well-known facts which follow from the {\it dictionnaire fonctiones--faisceaux}.
\begin{enumerate}
\item[(i)]
If $\sch{f} : \sch{X}\to \sch{Y}$ is a morphism of finite type in the category of schemes over $\kkq$ and defined over $\kq$, and if $\fais{B}\in \obj D^b_c(\sch{Y},\EE)$ is 
Frobenius-stable and $\varphi_\fais{B} : \Frob_{\sch{Y}}^*\fais{B} \to \fais{B}$ is an isomorphism, then $\sch{f}^*\fais{B}$ is Frobenius-stable and there is a canonical isomorphism 
$\varphi_{\sch{f}^*\fais{B}} : \Frob_{\sch{X}}^*\sch{f}^*\fais{B} \to \sch{f}^*\fais{B}$ and 
\begin{equation}\label{equation: pb and f}
\chf{\varphi_{\sch{f}^*\fais{B}}}(a) = \chf{\varphi_\fais{B}}\left(\sch{f}(a)\right)
\end{equation}
for each $a\in \sch{X}(\kq)$ (\cf \cite[1.1.1.4]{Lau}). 
\item[(ii)]
Likewise, if $\sch{f} : \sch{X}\to \sch{Y}$ is a morphism of finite type in the category of schemes over $\kkq$ and defined over $\kq$, and if $\fais{A}\in \obj D^b_c(\sch{X},\EE)$ is 
Frobenius-stable and $\varphi_\fais{A} : \Frob_{\sch{X}}^*\fais{A} \to \fais{A}$ is an isomorphism,  then $\sch{f}_!\fais{A}$ is Frobenius-stable and there is a canonical isomorphism 
$\varphi_{\sch{f}_!\fais{A}} : \Frob_{\sch{Y}}^*\, \sch{f}_!\fais{A} \to \sch{f}_!\fais{A}$ and 
\begin{equation}
\chf{\varphi_{\sch{f}_!\fais{A}}}(b) = \sum_{a\in \kksch{f}^{-1}(b)}\chf{\varphi_\fais{A}}(a)
\end{equation}
for each $b\in \sch{Y}(\kq)$ (\cf \cite[1.1.1.3]{Lau}).
\item[(iii)]
Likewise, if $\fais{A}\in \obj D^b_c(\sch{X},\EE)$ is 
Frobenius-stable and $\varphi_\fais{A} : \Frob_{\sch{X}}^*\fais{A} \to \fais{A}$ is an isomorphism,  then the Tate twist $\fais{A}(n)$ of $\fais{A}$ is Frobenius-stable, there is a canonical isomorphism 
$\varphi_{\fais{A}(n)} : \Frob_{\sch{X}}^*\fais{A}(n) \to \fais{A}(n)$ and 
\begin{equation}
\chf{\varphi_{\fais{A}(n)}}= q^{-n}\chf{\varphi_\fais{A}}.
\end{equation}
(\cf \cite[1.1.1.0]{Lau})
\end{enumerate}

Now, let $\ksch{G}$ be a connected reductive group over $\kq$ and set $\kksch{G}= \ksch{G}\times_\Spec{\kq} \Spec{\kkq}$. Let $\kksch{P}$ be a parabolic subgroup of $\kksch{G}$ with Levi component $\kksch{L}$ defined over $\kq$. Then there is a canonical isomorphism $\Frob_{\kksch{L}}^*\ \res^{\kksch{G}}_{\kksch{P}} \fais{A} \iso \res^{\kksch{G}}_{\kksch{P}}\ \Frob_{\kksch{G}}^*\fais{A}$ for each $\fais{A}\in \obj\catM \kksch{G}$. Moreover, if  
$\varphi_\fais{A} : \Frob_{\kksch{G}}^*\ \fais{A} \to \fais{A}$ is an isomorphism then $\res^{\kksch{G}}_{\kksch{P}}\fais{A}$ is also Frobenius-stable with respect to the isomorphism 
\[
\cdef{\varphi_{\res^{\kksch{G}}_{\kksch{P}}\fais{A}} : \Frob_{\kksch{L}}^* \res^{\kksch{G}}_{\kksch{P}}
\fais{A} \to \res^{\kksch{G}}_{\kksch{P}}\fais{A}}
\]
defined by the following diagram.
\begin{equation}\label{equation: FS restriction}
\xymatrix{
\Frob_{\kksch{L}}^* \res^{\kksch{G}}_{\kksch{P}} \fais{A} \ar[d]^{\iso} \ar@{.>}[rr]^{\varphi_{\res^{\kksch{G}}_{\kksch{P}}\fais{A}}} &&\ar@{=}[d] \res^{\kksch{G}}_{\kksch{P}} 
\fais{A} \\
\res^{\kksch{G}}_{\kksch{P}}\ \Frob_{\kksch{G}}^*  \fais{A} \ar[rr]^{\res^{\kksch{G}}_{\kksch{P}}\varphi_\fais{A}} && \res^{\kksch{G}}_{\kksch{P}} \fais{A}
}
\end{equation}
Note that it is \emph{not} necessary that $\kksch{P}$ be defined over $\kq$! However, if it is the case that $\kksch{P}$ is defined over $\kq$ --- in which case $\kksch{P} = \ksch{P}\times_\Spec{\kq} \Spec{\kkq}$ --- and if $\ksch{P}$ is a parabolic subgroup of $\ksch{G}$, then it follows from the properties of the {\it dictionnaire fonctiones--faisceaux} reviewed above that
\begin{equation}\label{equation: restriction and characteristic functions}
\chf{\varphi_{\res^{\kksch{G}}_{\kksch{P}}\fais{A}}} = \Res^{\ksch{G}(\kq)}_{\ksch{P}(\kq)}\ \chf{\varphi_\fais{A}},
\end{equation}
where $\pi_\ksch{P}$ is defined as in Subsection \ref{subsection: restriction} and
$\Res^{\ksch{G}(\kq)}_{\ksch{P}(\kq)}$ refers to normalised parabolic restriction 
at the level of functions (\ie averaging on fibres of the reductive quotient map $\ksch{P}(\kq)\to\ksch{L}(\kq)$).

Likewise, with $\kksch{G}$, $\kksch{P}$ and $\kksch{L}$ as above, if $\fais{B}$ is a character sheaf on $\kksch{L}$ then there is a canonical isomorphism $\Frob_{\kksch{L}}^*\ \ind^{\kksch{G}}_{\kksch{P}}\fais{B} \iso \ind^{\kksch{G}}_{\kksch{P}}\ \Frob_{\kksch{L}}^*\fais{B}$. (Induction of character sheaves is reviewed briefly in Subsection~\ref{subsection: induction}.) Moreover, if $\fais{B}$ is equipped with an isomorphism $\varphi_\fais{B} : \Frob_{\kksch{L}}^*\ \fais{B} \mathop{\to}\limits^{\iso} \fais{B}$, there is a \emph{canonical} isomorphism 
\[
\cdef{\varphi_{\ind^{\kksch{G}}_{\kksch{P}} \fais{B}} : 
\Frob_{\kksch{G}}^*\  \ind^{\kksch{G}}_{\kksch{P}} \fais{B} \mathop{\to}\limits^{\iso} \ind^{\kksch{G}}_{\kksch{P}} \fais{B}}
\]
Again, $\kksch{P}$ need not be defined over $\kq$; however, if it is the case that $\kksch{P}$ is defined over $\kq$, and if $\ksch{P}$ is a parabolic subgroup of $\ksch{G}$, then
\begin{equation}\label{equation: induction and characteristic functions}
\chf{\varphi_{\ind^{\kksch{G}}_{\kksch{P}} \fais{B}}} = \Ind^{\ksch{G}(\kq)}_{\ksch{P}(\kq)}\ 
\chf{\varphi_\fais{B}},
\end{equation}
where $\Ind^{\ksch{G}(\kq)}_{\ksch{P}(\kq)}$ refers to parabolic induction at the level of functions. 

\section{Nearby Cycles and Perverse Sheaves}\label{section: restriction and conjugation and induction}

In this section we prove three theorems (Theorems~\ref{theorem: restriction}, \ref{theorem: conjugation} and \ref{theorem: induction}) concerning nearby cycles of character sheaves. Although this section is quite technical, it forms the heart of the paper.

Recall that we fixed a Henselian trait $\sch{S}$ in Section~\ref{section: basic notions}. In order to simplify notation slightly, we now write \cdef{$\Kq$} for the non-Archimedean local field $\mathcal{O}_{\sch{S}}({\eta})$ and \cdef{$\KKq$} for the algebraically closed field $\mathcal{O}_{\sch{S}}({\bar \eta})$. Thus, $\mathcal{O}_{\sch{S}}(\sch{S})$ is the ring of integers \cdef{$\Rq$} of $\Kq$, and $\mathcal{O}_{\sch{\bar S}}(\sch{\bar S})$ is the ring of integers \cdef{$\RKK$} of $\KKq$. Recall (from Subsection~\ref{subsection: characteristic function}) that we write \cdef{$\kq$} for the residue field $\mathcal{O}_{\sch{S}}({\bf s})$ and \cdef{$\kkq$} for the algebraically closed field $\mathcal{O}_{\sch{S}}({\bar{\bf s}})$.

\subsection{Families of Integral Models}\label{subsection: models}

Let \cdef{$\Ksch{G}$} be a connected, reductive algebraic group over $\Kq$.
Let \cdef{$I(\Ksch{G},\Kq)$} be the extended Bruhat-Tits building of $\Ksch{G}(\Kq)$.
The cells of the CW-structure for $I(\Ksch{G},\Kq)$ are polysimplices, commonly called facets (as in \cite{SS}, for example). The minimal facet containing $x$ (an element of $I(\Ksch{G},\Kq)$) will be denoted by \cdef{$(x)$}. For each $x\in I(\Ksch{G},\Kq)$, the parahoric subgroup 
of Bruhat-Tits will be denoted by \cdef{$\Ksch{G}(\Kq)_x$} (\cf \cite[4.6.28]{BT2}). 

Let \cdef{$\RSch{G}{x}$} be the integral model of $\Ksch{G}$ associated to $x$ (\cf \cite[5.1.30, 5.2.1]{BT2} and \cite[\S7]{Yu}); thus, $\RSch{G}{x}$ is a connected, smooth affine group scheme defined over $\Rq$ with generic fibre $\Ksch{G}$ and $\RSch{G}{x}(\Rq)=\Ksch{G}(\Kq)_x$ (\cf \cite[7.2]{Yu}). 

The following lemma will be used in the proof of Theorem~\ref{theorem: induction}.

\begin{lemma}\label{lemma: unramified model}
If $\Lq/\Kq$ is a finite unramified extension and if $x$ is an element of $I(\Ksch{G},\Kq)$, then $\LBSch{G}{x'} = \KSch{G}{x} \times_\Spec{\Rq} \Spec{\RLq}$, where $x'$ is the image of $x$ under $I(\Ksch{G},\Kq) \hookrightarrow I(\Lsch{G},\Lq)$.
\end{lemma}

\begin{proof}
Use etale descent, as in \cite[10.9]{Lan1}; see also \cite[2.4]{Yu}.
\end{proof}


The special fibre \cdef{$\kSch{G}{x}$} of $\RSch{G}{x}$ is also a connected smooth affine group scheme over $\kkq$ (\cf \cite[7.2]{Yu}). 
(Recall that an integral scheme is {\em connected} if its generic fibre and closed fibre are connected (\cf \cite[1.2]{Yu}).)

Although $\kSch{G}{x}$ is an algebraic group over $\kq$, and therefore reduced as a scheme over $\kq$, it need not be a reductive algebraic group. Let \cdef{$\rkSch{G}{x}$} be the maximal reductive quotient of $\kSch{G}{x}$ and let 
\[
\cdef{{\nu_{\RSch{G}{x}}} : \kSch{G}{x} \to \rkSch{G}{x}}
\]
be the quotient morphism. Then $\rkSch{G}{x}$ is a connected reductive linear algebraic group over $\kq$. 


Extending scalars yields the following diagram.
\begin{equation}\label{diagram: models}
\xymatrix{
    \KKsch{G} \ar@{=}[r] & \KKSch{G}{x} \ar[r]^{j_{\RRSch{G}{x}}} \ar[d] & \RRSch{G}{x} \ar[d] & \ar[l]_{i_{\RRSch{G}{x}}} \kkSch{G}{x} \ar[d] \ar[r]^{\nu_{\RRSch{G}{x}}} & \rkkSch{G}{x}  \\
&    \{\bar \eta\} \ar[r] & \sch{\bar S} & \ar[l] \{\bar{\bf s}\} & \\
}
\end{equation}

\begin{lemma}\label{lemma: even}
For each $x\in I(\Ksch{G},\Kq)$, the dimension of $\nu_{\RRSch{G}{x}}$ is even.
\end{lemma}

\begin{proof}
Let $x$ and $y$ be elements of $I(\Ksch{G},\Kq)$. We say that \cdef{$x\leq y$ in the Bruhat partial order for $I(\Ksch{G},\Kq)$} when the closure $\overline{(x)}$ of the minimal facet $(x)$ containing $x$ is contained in the closure $\overline{(y)}$ of the minimal facet $(y)$ containing $y$. 

Suppose $x\leq y$. Let \cdef{$\RSch{a}{x\leq y} : \RSch{G}{y} \to \RSch{G}{x}$} be the morphism of group schemes over $\Rq$ obtained by extending the identity morphism $\id_\Ksch{G}$ in the category of group schemes over $\sch{S}$. (This is denoted by $\operatorname{Res}^{F'}_{F}$ in \cite[9.21]{Lan1} with $F'=(y)$ and $F=(x)$; see also  \cite[6.2]{Lan1}.) By restriction to special fibres, this defines a morphism \cdef{$\kSch{a}{x\leq y} : \kSch{G}{y} \to \kSch{G}{x}$} of group schemes over $\kq$. 
\begin{equation}\label{equation: vanishing square}
	\xymatrix{
	\ar[d]^{\id} \Ksch{G}  \ar[r]^{j_{\RSch{G}{y}}} & \RSch{G}{y} \ar[d]^{\RSch{a}{x\leq y}} & \ar[l]_{i_{\RSch{G}{y}}} \kSch{G}{y}  \ar[d]^{\kSch{a}{x\leq y}} \\
	\ar[d] \Ksch{G} \ar[r]^{j_{\RSch{G}{x}}} & \ar[d] \RSch{G}{x} & \ar[d]\ar[l]_{i_{\RSch{G}{x}}} \kSch{G}{x}\\
\{ \eta \} \ar[r] &  \sch{S} & \ar[l] \{ {\bf s}\}
	}
\end{equation}
All squares in the diagram above are Cartesian. 

Let $\rkPsch{G}{x\leq y}$ be the schematic image of the morphism of algebraic groups $\ksch{\nu}_{x} \circ  \kSch{a}{x\leq y}: \kSch{G}{y} \to \rkSch{G}{x}$; this homomorphism is denoted by $\operatorname{\overline{Res}}^{F'}_{F}$ in \cite[9.21]{Lan1} with $F'=(y)$ and $F=(x)$ and its schematic image is denoted by $p(F')$ in \cite[9.22]{Lan1}. By \cite[6.7-6.9, 9.22]{Lan1}, $\rkPsch{G}{x\leq y}$ is a parabolic subgroup of $\rkSch{G}{x}$ with reductive quotient $\rkSch{G}{y}$.
\[
\xymatrix{
\ar[d]^{\nu_{\RSch{G}{y}}}  \kSch{G}{y} \ar[rr]^{\kSch{a}{x\leq y}} \ar[drr]^{\nu_{\RSch{G}{x}}\circ \kSch{a}{x\leq y}}  \ar[rd] && \kSch{G}{x} \ar[d]^{\nu_{\RSch{G}{x}}} \\
\rkSch{G}{y}  & \rkPsch{G}{x\leq y} \ar[r] & \rkSch{G}{x} \\
}
\]

Let \cdef{$\kPsch{G}{x\leq y}$} be the schematic image of $\kSch{a}{x\leq y}$.
%
%
Factor $\kSch{a}{x\leq y}$ by
\begin{equation}\label{equation: iso thm}
	\kSch{a}{x\leq y} = \ksch{h}_{x\leq y} \circ \ksch{g}_{x\leq y}
\end{equation}
where \cdef{$\ksch{g}_{x\leq y}$} is an epimorphism onto $\kPsch{G}{x\leq y}$ and \cdef{$\ksch{h}_{x\leq y}$} is a closed immersion into $\kSch{G}{x}$.
Let \cdef{$\ksch{\iota}_{x\leq y} : \rkPsch{G}{x\leq y} \hookrightarrow \rkSch{G}{x}$} be the obvious closed immersion and let
\begin{equation}\label{equation: r}
\cdef{\ksch{\tau}_{x\leq y}: \rkPsch{G}{x\leq y} \to \rkSch{G}{y}}.
\end{equation}
be the reductive quotient map. 
%
Then since $\rkSch{G}{y}$ is a Levi subgroup of $\rkSch{G}{x}$, we have  
\begin{equation}\label{equation: dxy}
\dim\ksch{\tau}_{x\leq y} = \frac{1}{2} \left( \dim \nu_{\RRSch{G}{y}} - \dim \nu_{\RRSch{G}{x}}\right).
\end{equation}

Next, notice that
\begin{equation}\label{equation: factor}
\nu_{\RSch{G}{x}}\circ \ksch{h}_{x\leq y} = \ksch{\iota}_{x\leq y} \circ \ksch{\nu}_{x\leq y}.
\end{equation}
where \cdef{$\ksch{\nu}_{x\leq y}: \kPsch{G}{x\leq y} \to \rkPsch{G}{x\leq y}$} is $\nu_{\RSch{G}{x}}\vert_{\kPsch{G}{x\leq y}}$ with restricted codomain.
Notice also that
\begin{equation}\label{equation: nuj}
\nu_{\RSch{G}{y}} = \ksch{\tau}_{x\leq y} \circ \ksch{\nu}_{x\leq y} \circ \ksch{g}_{x\leq y}.
\end{equation}
Diagram~\eqref{diagram : the map tauxy - Hadi} visualizes the relations between these maps.
\begin{equation}
\label{diagram : the map tauxy - Hadi}
\xymatrix{
	\kSch{G}{y} \ar[rr]^{\kSch{a}{x\leq y}} \ar[dr]^{\ksch{g}_{x\leq y}} \ar[d]_{\nu_{\RSch{G}{y}}} && \kSch{G}{x} \ar[d]^{\nu_{\RSch{G}{x}}}\\
 	\rkSch{G}{y} & \kPsch{G}{x\leq y} \ar[ur]^{\ksch{h}_{x\leq y}} \ar[d]^{\ksch{\nu}_{x\leq y}} & \rkSch{G}{x}\\
	& \rkPsch{G}{x\leq y} \ar[ul]^{\ksch{\tau}_{x\leq y}} \ar[ur]_{\ksch{\iota}_{x\leq y}} & \\
}
\end{equation}

Consider the following special case: if $x_0$ is contained in a hyperspecial (poly)vertex, then $\kSch{G}{x_0} = \rkSch{G}{x_0}$ so $\dim \nu_{\RSch{G}{x_0}} =0$. In particular, $\dim \nu_{\RSch{G}{x_0}}$ is even.

Let $x$ be arbitrary. Let $y$ be a point in a big cell for $I(\Ksch{G},\Kq)$ such that $x\leq y$. Let $x_0$ be a hyperspecial point in the closure of $y$. Then $x_0 \leq y \geq x$. Now \eqref{equation: dxy} shows that $\dim\nu_{\RSch{G}{x_0}}$ has the same parity as $\dim\nu_{\RSch{G}{y}}$, and also that $\dim\nu_{\RSch{G}{y}}$ has the same parity as $\dim\nu_{\RSch{G}{x}}$. Thus, $\dim\nu_{\RSch{G}{x}}$ has the same parity as $\dim\nu_{\RSch{G}{x_0}}$. Since we have seen that $\dim\nu_{\RSch{G}{x_0}}$ is even, it follows that $\dim\nu_{\RSch{G}{x}}$ is even.
\end{proof}

\begin{definition}\label{definition: the functor}
For $x\in I(\Ksch{G},\Kq)$, 
define the functor 
$$
\RES{G}{x} : D^b_c(\KKsch{G},\EE) \to D^b_c(\rkkSch{G}{x},\EE)
$$ 
by
\[
\RES{G}{x} \ceq {\nu_{\RRSch{G}{x}}}_!\ (\dim \nu_{\RRSch{G}{x}}/2)\ \NC{\RRSch{G}{x}},
\]
where $(\dim \nu_{\RRSch{G}{x}}/2)$ denotes Tate twist by $\dim \nu_{\RRSch{G}{x}}/2$.
\end{definition}

Observe that Lemma~\ref{lemma: even} is required to see that this definition makes sense.

The functor $\RES{G}{x}$ plays a crucial role in this paper.

\subsection{Restriction at the Level of Reductive Quotients}\label{subsection: rrq}

Now we come to Theorem~\ref{theorem: restriction}, which is one of the main results of this paper as it makes it possible to determine $\RESFAIS{F}{G}{y}$ from $\RESFAIS{F}{G}{y}$ when $x\leq y$ in $I(\Ksch{G},\Kq)$ in the Bruhat partial order for $I(\Ksch{G},\Kq)$, for each $\fais{F}\in \obj D^b_c(\KKsch{G},\EE)$.

\begin{theorem}\label{theorem: restriction}
If $x$ and $y$ are elements of $I(\Ksch{G},\Kq)$ and $x\leq y$ in the Bruhat partial order, then there is a parabolic subgroup $\rkPsch{G}{x\leq y} \subset \rkSch{G}{x}$ with reductive quotient $\rkSch{G}{y}$ and a natural isomorphism of functors $D^b_c(\KKsch{G},\EE) \to D^b_c(\rkkSch{G}{y},\EE)$
\[
 \res^{\rkkSch{G}{x}}_{\rkkPsch{G}{x\leq y}} \RES{G}{x} \to \RES{G}{y},
\]
where $\res^{\rkkSch{G}{x}}_{\rkkPsch{G}{x\leq y}}$ is the parabolic restriction functor of Subsection~\ref{subsection: restriction} (\cf \eqref{equation: restriction finite} in particular).
\end{theorem}

\begin{proof}

The parabolic subgroup $\rkPsch{G}{x\leq y}$ promised in the statement of Theorem~\ref{theorem: restriction} is defined in the proof of Lemma~\ref{lemma: even}.
%

Returning to \eqref{diagram : the map tauxy - Hadi}, extend scalars from $\mathcal{O}_\sch{S}({\bf s}) = \kq$ to $\mathcal{O}_\sch{\bar S}({\bar{\bf s}})=\kkq$ and consider the following diagram, in which we take the liberty of denoting some morphisms over $\kkq$ by the same symbol used to indicate the corresponding morphism before extending scalars.
\[
\xymatrix{
	\kkSch{G}{y} \ar[rr]^{\kkSch{a}{x\leq y}} \ar[dr]^{\ksch{g}_{x\leq y}} \ar[d]_{\nu_{\RRSch{G}{y}}} && \kkSch{G}{x} \ar[d]^{\nu_{\RRSch{G}{x}}}\\
 	\rkkSch{G}{y} & \kkPsch{G}{x\leq y} \ar[ur]^{\ksch{h}_{x\leq y}} \ar[d]^{\ksch{\nu}_{x\leq y}} & \rkkSch{G}{x}\\
	& \rkkPsch{G}{x\leq y} \ar[ul]^{\ksch{\tau}_{x\leq y}} \ar[ur]_{\ksch{\iota}_{x\leq y}} & \\
}
\]

Now we define the natural isomorphism promised in Theorem~\ref{theorem: restriction}.
Recalling the definition of $\res^{\rkkSch{G}{x}}_{\rkkPsch{G}{x\leq y}}$ from Subsection~\ref{subsection: restriction} together with Definition~\ref{definition: the functor} 
we have 
\begin{equation}\label{equation: restriction 1}
\begin{aligned}
& \res^{\rkkSch{G}{x}}_{\rkkPsch{G}{x\leq y}}\ \RES{G}{x}\\
&= {\ksch{\tau}_{x\leq y}}_!\ (\dim\tau_{x\leq y})\ {\ksch{\iota}_{x\leq y}}^*\ {\nu_{\RRSch{G}{x}}}_!\ (\dim \nu_{\RRSch{G}{x}}/2)\ \NC{\RRSch{G}{x}}.
\end{aligned}
\end{equation}
By \eqref{equation: dxy}, 
\begin{equation}
\begin{aligned}
& {\ksch{\tau}_{x\leq y}}_!\ (\dim\tau_{x\leq y})\ {\ksch{\iota}_{x\leq y}}^*\ {\nu_{\RRSch{G}{x}}}_!\ (\dim \nu_{\RRSch{G}{x}}/2)\ \NC{\RRSch{G}{x}} \\
&= {\ksch{\tau}_{x\leq y}}_!\ \ {\ksch{\iota}_{x\leq y}}^*\ {\nu_{\RRSch{G}{x}}}_!\ (\dim \nu_{\RRSch{G}{y}}/2)\ \NC{\RRSch{G}{x}} .
\end{aligned}
\end{equation}
Now,  by construction, the generic fibre of $\RSch{a}{x\leq y}$ is $\id : \Ksch{G}\to \Ksch{G}$. 
Since $\RSch{a}{x\leq y}$ is a closed immersion, it is proper; the same is true of $\RRSch{a}{x\leq y}$. Therefore, it follows from Subsection~\ref{item: FNC_!} that there is a natural isomorphism
\begin{equation}\label{equation: crux}
{\kkSch{a}{x\leq y}}_!\  \NC{\RRSch{G}{y}} \iso \NC{\RRSch{G}{x}}.
\end{equation}
Thus,
\begin{equation}
\begin{aligned}
& {\ksch{\tau}_{x\leq y}}_!\ \ {\ksch{\iota}_{x\leq y}}^*\ {\nu_{\RRSch{G}{x}}}_!\ (\dim \nu_{\RRSch{G}{y}}/2)\ \NC{\RRSch{G}{x}} \\
 & \iso {\ksch{\tau}_{x\leq y}}_!\ {\ksch{\iota}_{x\leq y}}^*\ {\nu_{\RRSch{G}{x}}}_!\  (\dim \nu_{\RRSch{G}{y}}/2)\  {\kSch{a}{x\leq y}}_!\  \NC{\RRSch{G}{y}}.
 \end{aligned}
\end{equation}
Since $\kSch{a}{x\leq y} = \ksch{h}_{x\leq y} \circ \ksch{g}_{x\leq y}$ (\cf \eqref{equation: iso thm}), we have 
\begin{equation}
\begin{aligned}
&{\ksch{\tau}_{x\leq y}}_!\ {\ksch{\iota}_{x\leq y}}^*\ {\nu_{\RRSch{G}{x}}}_!\  {\kSch{a}{x\leq y}}_!\ (\dim \nu_{\RRSch{G}{y}}/2)\ \NC{\RRSch{G}{y}} \\
&\iso {\ksch{\tau}_{x\leq y}}_!\ {\ksch{\iota}_{x\leq y}}^*\ {\nu_{\RRSch{G}{x}}}_!\  {\ksch{h}_{x\leq y}}_! \ {\ksch{g}_{x\leq y}}_!\  (\dim \nu_{\RRSch{G}{y}}/2)\ \NC{\RRSch{G}{y}}.
\end{aligned}
\end{equation}
Since $\nu_{\RSch{G}{x}}\circ \ksch{h}_{x\leq y} = \ksch{\iota}_{x\leq y}\circ \ksch{\nu}_{x\leq y}$ (\cf \eqref{equation: factor}), we have
\begin{equation}
\begin{aligned}
& {\ksch{\tau}_{x\leq y}}_!\ {\ksch{\iota}_{x\leq y}}^*\ {\nu_{\RRSch{G}{x}}}_!\  {\ksch{h}_{x\leq y}}_! \ {\ksch{g}_{x\leq y}}_!\ (\dim \nu_{\RRSch{G}{y}}/2)\  \NC{\RRSch{G}{y}}\\
&\iso {\ksch{\tau}_{x\leq y}}_!\ {\ksch{\iota}_{x\leq y}}^*\ {\ksch{\iota}_{x\leq y}}_!\ {\nu_{x\leq y}}_!\ {\ksch{g}_{x\leq y}}_!\  (\dim \nu_{\RRSch{G}{y}}/2)\ \NC{\RRSch{G}{y}}.
\end{aligned}
\end{equation}
Now, $\ksch{\iota}_{x\leq y} : \rkPsch{G}{x\leq y} \to \rkSch{G}{x}$ is a closed immersion, so ${\ksch{\iota}_{x\leq y}}_* \iso {\ksch{\iota}_{x\leq y}}_!$ and the adjunction 
\begin{equation}\label{equation: identity}
{\ksch{\iota}_{x\leq y}}^*\ {\ksch{\iota}_{x\leq y}}_* \to \id
\end{equation}
is also an isomorphism \cite[1.4.1.2]{BBD}. Thus,
\begin{equation}
\begin{aligned}
&{\ksch{\tau}_{x\leq y}}_!\ {\ksch{\iota}_{x\leq y}}^*\ {\ksch{\iota}_{x\leq y}}_!\ {\nu_{x\leq y}}_!\ {\ksch{g}_{x\leq y}}_!\  (\dim \nu_{\RRSch{G}{y}}/2)\ \NC{\RRSch{G}{y}}\\
& \iso {\ksch{\tau}_{x\leq y}}_!\ {\nu_{x\leq y}}_!\ {\ksch{g}_{x\leq y}}_!\ (\dim \nu_{\RRSch{G}{y}}/2)\ \NC{\RRSch{G}{y}}.
\end{aligned}
\end{equation}
Finally, recall from \eqref{equation: nuj} that $\nu_{\Rsch{G}_y} = \ksch{\tau}_{x\leq y} \circ \nu_{x\leq y} \circ \ksch{g}_{x\leq y}$; thus,
\[
{\ksch{\tau}_{x\leq y}}_!\  {\nu_{x\leq y}}_!\ {\ksch{g}_{x\leq y}}_!\ (\dim \nu_{\RRSch{G}{y}}/2)\ \NC{\RRSch{G}{y}}
\iso {\nu_{\Rsch{G}_y}}_!\ (\dim \nu_{\RRSch{G}{y}}/2)\ \NC{\RRSch{G}{y}} .
\]
This concludes the proof of Theorem~\ref{theorem: restriction}.
\end{proof}


\begin{remark}\label{remark: transitive}
Theorem~\ref{theorem: restriction} shows that there is a parabolic subgroup $\rkPsch{G}{x\leq y} \subset \rkSch{G}{x}$ with reductive quotient $\rkSch{G}{y}$ and a natural transformation
\[
R_{x\leq y} : \res^{\rkkSch{G}{x}}_{\rkkPsch{G}{x\leq y}}\ \RES{G}{x} \to \RES{G}{y},
\]
for each $x\leq y$ in $I(\Ksch{G},\Kq)$. In fact, more is true: if $x\leq y\leq z$, then the following diagramme commutes.
\[
\xymatrix{
\res^{\rkkSch{G}{y}}_{\rkkPsch{G}{y\leq z}} \res^{\rkkSch{G}{x}}_{\rkkPsch{G}{x\leq y}}
\RES{G}{x} \ar[r]^{\iso} \ar[d]^{\res^{\rkkSch{G}{y}}_{\rkkPsch{G}{y\leq z}}}
& 
\res^{\rkkSch{G}{x}}_{\rkkPsch{G}{x\leq z}} 
\RES{G}{x}  \ar[d]^{R_{x\leq z}} \\
\res^{\rkkSch{G}{y}}_{\rkkPsch{G}{y\leq z}} 
\RES{G}{y}  \ar[r]^{R_{y\leq z}}
& 
\RES{G}{z}
}
\]
Here the isomorphism on the top is given by the transitivity of the parabolic restriction functor (\cf \cite[(15.3.3)]{CS} or \cite[(8)]{MS}) and the fact that $\dim\ksch{\tau}_{x\leq z} = \dim\ksch{\tau}_{x\leq y} + \dim\ksch{\tau}_{y\leq z}$ (use \eqref{equation: dxy}).
\end{remark}

\subsection{Conjugation and Nearby Cycles}\label{subsection: conjugation}

Let \cdef{$\Ksch{m} : \Ksch{G}\times\Ksch{G} \to \Ksch{G}$} be conjugation over $\mathcal{O}_S(\eta) = \Kq$ given by $\Ksch{m}(g,h) = ghg^{-1}$. The Bruhat-Tits building $I(\Ksch{G},\Kq)$ is equipped with an action by $\Ksch{G}(\Kq)$ which we indicate by
\begin{eqnarray*}
\Ksch{G}(\Kq) \times I(\Ksch{G},\Kq) &\to& I(\Ksch{G},\Kq)\\
    (g,x) &\mapsto& gx.
\end{eqnarray*}
Let \cdef{$\tKq$}  be the maximal tamely ramified extension of $\Kq$ in $\KKq$.
Suppose $g\in \Ksch{G}(\tKq)$. Then $g\in \Ksch{G}(\Lq)$ for some finite, tamely ramified Galois extension $\Lq/\Kq$. Then $\Gal(\Lq/\Kq)$ acts on $I(\Lsch{G},\Lq)$ and $I(\Ksch{G},\Kq) = I(\Lsch{G},\Lq)^{\Gal(\Lq/\Kq)}$ (\cf \cite{P}). 

\begin{theorem}\label{theorem: conjugation}
Fix $x\in I(\Ksch{G},\Kq)$ and $g\in \Ksch{G}(\tKq)$; suppose $gx \in I(\Ksch{G},\Kq)$.
There is a canonical isomorphism $\rkkSch{m(g^{-1})}{gx}: \rkkSch{G}{gx} \to \rkkSch{G}{x}$
and if $\fais{F}$ is an equivariant perverse sheaf on $\KKsch{G}$, then
\[
(\rkkSch{m(g^{-1})}{gx})^\ast\ \RES{G}{x} \fais{F} \iso \RES{G}{gx}\ \fais{F}.
\]
\end{theorem}

\begin{proof}
We begin by defining $\rkkSch{m(g^{-1})}{gx}: \rkkSch{G}{gx} \to \rkkSch{G}{x}$.
Let $\KKsch{m} : \KKsch{G}\times \KKsch{G} \to \KKsch{G}$ be the morphism obtained from $\Ksch{m}$ (introduced just before the statement of Theorem~\ref{theorem: conjugation}) by extending scalars from $\mathcal{O}_S(\eta)=\Kq$ to $\mathcal{O}_{\bar S}({\bar \eta})=\KKq$. For each $g\in \Ksch{G}(\KKq)$, let \cdef{$\KKsch{m(g)} : \KKsch{G} \to \KKsch{G}$} be the morphism given by $\KKsch{m(g)}(h) = \KKsch{m}(g,h)$ for $h \in \Ksch{G}(\KKq)$.

As above, let $\Lq/\Kq$ be a finite, tamely ramified, Galois extension 
such that $g\in \Ksch{G}(\Lq)$.
By the Extension Principle \cite[2.3]{Yu}\footnote{The proof of the Extension Principle of \cite[2.3]{Yu} is made using \cite[1.7. Sch\'emas \'etoff\'es]{BT2}; in fact, it is almost exactly \cite[Prop.~1.7.6]{BT2}.}, the isomorphism $\Lsch{m(g)}:\Lsch{G} \to \Lsch{G}$ of group schemes over $\Lq$ extends uniquely to an isomorphism 
$$
\RLSch{m(g)}{x}: \RLSch{G}{x} \to \RLSch{G}{gx}
$$ 
of group schemes over $\RLq$. 
The Extension Principle
is applicable because $\RLSch{G}{x}$ and $\RLSch{G}{gx}$ are {\em smooth} integral models of $\Lsch{G}$ and $\Lsch{m(g)}((\RLSch{G}{x})(\RLq))= (\RLSch{G}{gx})(\RLq)$. 

By base extension, $\RLSch{m(g)}{x}$ extends to an isomorphism 
$\RRSch{m(g)}{x}: \RRSch{G}{x} \to \RRSch{G}{gx}$.
Restricting to special fibres and thence to reductive quotients, the isomorphism $\RRSch{m(g)}{x}$ 
defines an isomorphism \cdef{$\rkkSch{m(g)}{x}: \rkkSch{G}{x} \to \rkkSch{G}{gx}$}. The diagram 
\begin{equation}\label{equation: c&vc}
	\xymatrix{
	\ar[d]^{\nu_{\RRSch{G}{x}}} \kkSch{G}{x} \ar[rr]^{\kkSch{m(g)}{x}} && \ar[d]^{\nu_{\RRSch{G}{gx}}} \kkSch{G}{gx} \\
	\rkkSch{G}{x} \ar[rr]^{\rkkSch{m(g)}{x}} && \rkkSch{G}{gx} 
	}
\end{equation}
is Cartesian.

Because $\fais{F}$ is $\KKsch{G}$-equivariant, it comes equipped with an (essentially unique) isomorphism $\mu_\fais{F} : \KKsch{m}^*\fais{F} \to \KKsch{\proj}^* \fais{F}$ in $D^b_c(\KKsch{G}\times \KKsch{G}, \EE)$, where $\Ksch{m} : \Ksch{G}\times \Ksch{G} \to \Ksch{G}$ is conjugation (\cf above) and $\Ksch{\proj}: \Ksch{G}\times \Ksch{G} \to \Ksch{G}$ is projection onto the second component (\cf Subsection~\ref{subsection: equivariant perverse sheaves}). Let
\begin{equation}\label{equation: equivariance at g}
\cdef{\mu_\fais{F}(g) : {\KKsch{m(g^{-1})}}^*\ \fais{F} \to \fais{F}}
\end{equation}
be the isomorphism in $D^b_c(\KKsch{G}, \EE)$ defined according to \eqref{equation: iii reductive}. Consider 
the diagram below.
\begin{equation}
\xymatrix{
\Ksch{G}_{\KKq} \ar@{=}[r] & \KKSch{G}{x} \ar[r]^{j_{\RRSch{G}{x}}} & \RRSch{G}{x} && \ar[ll]_{i_{\RRSch{G}{x}}} \kkSch{G}{x} \ar[rr]^{\nu_{\RSch{G}{x}}} && \rkkSch{G}{x} \\
\ar[u]^{\Ksch{m}_{\KKq}(g^{-1})} \KKsch{G} \ar@{=}[r] &  \KKSch{G}{gx} \ar[u]
\ar[r]_{j_{\RRSch{G}{gx}}} & \ar[u]^{\RRSch{m(g^{-1})}{gx}} \RRSch{G}{gx} && \ar[ll]^{i_{\RRSch{G}{gx}}} \ar[u]^{\kkSch{m(g^{-1})}{gx}} \kkSch{G}{gx} \ar[rr]_{\nu_{\RRSch{G}{gx}}} && \ar[u]^{\rkkSch{m(g^{-1})}{gx}} \rkkSch{G}{gx} \\
}
\end{equation}
The map 
\begin{equation}
\label{equation : first isomorphism - Hadi}
\NC{\RRSch{G}{gx}}\ \mu_\fais{F}(g):
\NC{\RRSch{G}{gx}}\ {\KKsch{m(g^{-1})}}^*\ \fais{F}\to
\NC{\RRSch{G}{gx}}\ \fais{F}
\end{equation}
is an isomorphism.
The morphism $\RSch{m(g^{-1})}{gx}$ is smooth, 
and smooth base change (Subsection \ref{item: SNC^*}) implies
\begin{equation}
\label{equation : second isomorphism - Hadi}
{\kkSch{m(g^{-1})}{gx}}^*\ \NC{\RRSch{G}{x}}\iso
\NC{\RRSch{G}{gx}}\ {\KKsch{m(g^{-1})}}^*\ \fais{F}.
\end{equation}
From \eqref{equation : first isomorphism - Hadi},
\eqref{equation : second isomorphism - Hadi}, and 
proper base change (Subsection \ref{item: PBC})
applied to the Cartesian diagram \eqref{equation: c&vc},
it follows that
\begin{equation*}
\begin{aligned}
&{\nu_{\RRSch{G}{gx}}}_!\ (\dim \nu_{\RRSch{G}{gx}}/2)\
\NC{\RRSch{G}{gx}}\ \fais{F} \\
&\iso
{\nu_{\RRSch{G}{gx}}}_!\ (\dim \nu_{\RRSch{G}{gx}}/2)\ {\kkSch{m(g^{-1})}{gx}}^*\ \NC{\RRSch{G}{x}}\fais{F}\\
&\iso
(\rkkSch{m(g^{-1})}{gx})^\ast\ {\nu_{\RRSch{G}{x}}}_!\ (\dim \nu_{\RRSch{G}{x}}/2)\ \NC{\RRSch{G}{x}}\ \fais{F}.
\end{aligned}
\end{equation*}
(Here we also used the fact that $\dim \nu_{\RRSch{G}{gx}} = \dim \nu_{\RRSch{G}{x}}$; \cf Definition~\ref{definition: the functor}.)
This concludes the proof of Theorem~\ref{theorem: conjugation}.
\end{proof} 

\subsection{Parabolic Induction and Nearby Cycles}\label{subsection: induction and nearby cycles}


In this subsection we prove Theorem~\ref{theorem: induction} which shows that if 
$\fais{F}$ is induced (in the sense of Subsection~\ref{subsection: induction}) from a character sheaf $\fais{G}$ of a maximal torus $\KKsch{T}$ defined over $\Kq$ and split over an unramified extension of $\Kq$, then we can determine $\RESFAIS{F}{G}{x}$ whenever $x\in I(\Ksch{G},\Kq)$ 
is a (poly)vertex in the image of $I(\Ksch{T},\Kq) \hookrightarrow I(\Ksch{G},\Kq)$. Note that the latter
map is defined by passing to an unramified splitting extension $\Lq$ of $\Kq$ for $T$; \ie
$$
I(\Ksch{T},\Kq)=I(\Ksch{T}_\Lq,\Lq)^{\mathrm{Gal}(\Lq/\Kq)}\hookrightarrow 
I(\Ksch{G}_\Lq,\Lq)^{\mathrm{Gal}(\Lq/\Kq)}=I(\Ksch{G},\Kq).
$$
First we state and prove the result for the case $T$ is split, and then show
how to extend it to the general case.

\begin{theorem}\label{theorem: induction}
Let $\Ksch{T}$ be a maximal split torus in $\Ksch{G}$, let $\LBsch{B}$ be a Borel subgroup of
$G$ with Levi factor $\Ksch{T}$. Suppose $(x)$ is a (poly)vertex in $I(\Ksch{T},\Kq)$ and that $(x)$ is a hyperspecial (poly)vertex in $I(\Ksch{G},\Kq)$. 
(With abuse of 
notation, the image of $x$
under the map $I(\Ksch{T},\Kq) \hookrightarrow I(\Ksch{G},\Kq)$ is denoted by $x$ as well.) 
Then there is a smooth integral model $\RLBSch{B}{x}$ for $\LBsch{B}$ such that 
$\RLBSch{B}{x}(\Kq) = \LBsch{B}(\Kq) \cap \RLBSch{G}{x}(\Kq)$, and  
the special fibre $\llBSch{B}{x}$ of $\RLBSch{B}{x}$ is a Borel subgroup of $\llBSch{G}{x}$.
Moreover, 
\[
\NC{\RRSch{G}{x}}\ \ind^{\KKsch{G}}_{\LLBsch{B}}\ \fais{G} \iso \ind^{\kkSch{G}{x}}_{\llSch{B}{x}}\ \NC{\RRSch{T}{x}}\ \fais{G},
\]
for every character sheaf $\fais{G}$ of $\KKsch{T}$.
Consequently,
\[
\RES{G}{x}\ \ind^{\KKsch{G}}_{\LLBsch{B}}\ \fais{G} \iso 
\ind^{\kkSch{G}{x}}_{\llSch{B}{x}}\ \RES{T}{x}\ \fais{G}.
\]
\end{theorem}

The proof of Theorem~\ref{theorem: induction} requires Lemmas~\ref{lemma: borel}, \ref{lemma: flag} and \ref{lemma: fibration} which will be given in Subsubsection~\ref{subsubsection: flag}. 

\subsubsection{Integral Models for Flag Varieties}\label{subsubsection: flag}

We suspect Lemmas~\ref{lemma: borel}, \ref{lemma: flag} and \ref{lemma: fibration} are well-known, but include them with proofs for completeness.



\begin{lemma}\label{lemma: borel}
The schematic closure \cdef{$\RLBSch{B}{x}$} of $\LBsch{B}$ in $\RLBSch{G}{x}$ 
is the unique smooth integral model (over $\sch{S}$) of $\LBsch{B}$  such that 
$\RLBSch{B}{x}(\Rq) = \LBsch{B}(\Kq) \cap \RLBSch{G}{x}(\Rq)$. Moreover, 
$\RLBSch{B}{x}$ is a closed subscheme of $\RLBSch{G}{x}$.
\end{lemma}

\begin{proof}
Observe that $\LBsch{B}$ is a closed subscheme of $\LBsch{G}$. By definition (\cf \cite[\S 2.6]{Yu}, for example) the schematic closure $\RLBSch{B}{x}$ of $\LBsch{B}$ in $\RLBSch{G}{x}$ is the smallest closed sub-scheme of $\RLBSch{G}{x}$ containing $\LBsch{B}$. By \cite[\S 2.6, Lemma]{Yu}, $\RLBSch{B}{x}$ is a model of $\LBsch{B}$ and $\RLBSch{B}{x}$ is a subscheme of $\RLBSch{G}{x}$. Moreover, $\LBSch{B}{x}$ is unique with the property $\RLBSch{B'}{x}(\RLq) = \LBsch{B}(\Lq) \cap \RLBSch{G}{x}(\Rq)$. 
To see why $\RLBSch{B}{x}$ is smooth, use 
\cite[\S 7, Theorem]{Yu} to see that
$\RLBSch{B}{x}$ is isomorphic (as a scheme over $\sch{S}$) to $\RLBSch{T}{x} \times \RLBSch{U_{B}}{x}$,
where $\RLBSch{U_{B}}{x}$ is the image of $\prod_{a\in {\Phi(\LBsch{G},\LBsch{T})}^+} \RLBSch{U_a}{x}$
under multiplication. Here $\Phi(\LBsch{G},\LBsch{T})^+$ is the set of positive roots 
of $\LBsch{G}$ with respect to $\LBsch{T}$ and \cdef{$\RLBSch{U_a}{x}$} is the unique smooth integral model of the root subgroup $\LBsch{U}_a \subset \LBsch{G}$ such that  $\RLBSch{U}{a}(\Rq) = \LBsch{U}_a(\Kq)_{x,0}$ (\cf \cite[\S4.3]{BT2}). Since $\RLBSch{T}{x}$ and 
$\RLBSch{U_{B}}{x}$ are smooth, and since the product is taken over $\sch{S}$, 
it follows that $\RLBSch{B}{x}$ is also smooth. 

\end{proof}

\begin{lemma}\label{lemma: flag}
There is a smooth principal fibration 
$\RLBSch{G}{x}\to \RLBSch{G}{x}/\RLBSch{B}{x}$ with group $\RLBSch{B}{x}$. Moreover, 
the generic fibre of $\RLBSch{G}{x}/\RLBSch{B}{x}$ is 
$\LBsch{G}/\LBsch{B}$, and its special
fibre is $\kSch{G}{x}/\kSch{B}{x}$, which is the flag variety of $\kSch{G}{x}$.
\end{lemma}

\begin{proof}
Lemma~\ref{lemma: flag} is a consequence of the Bruhat decomposition for $\RLBSch{G}{x}$, which we now sketch.
Let \cdef{$\Phi(\LBsch{G},\LBsch{T})_{x}$} (resp. \cdef{$\Phi(\LBsch{G},\LBsch{T})_{x}^+$}) be the set of roots $a\in \Phi(\LBsch{G},\LBsch{T})$ (resp. $a\in \Phi(\LBsch{G},\LBsch{T})^+$) for which $\alpha(x)=0$ where $\alpha$ is an affine root of $\LBsch{G}$ with vector part equal to $a$. Also, let \cdef{$W_{x} = 
W(\LBsch{G},\LBsch{T})_{x}$} be the Weyl group for the root system $\Phi(\LBsch{G},\LBsch{T})_{x}$. For each $w\in W_{x}$, define \cdef{$\Phi_{x}(w)^+ \ceq \{ a \in \Phi(\LBsch{G},\LBsch{T})_{x}^+ \tq w(a) \in \Phi(\LBsch{G},\LBsch{T})_{x}^-\}$}.
Write \cdef{$\RLBSch{U_w}{x}$} for image of $\prod_{a\in \Phi_{x}(w)^+} \RLBSch{U_a}{x}$ under the multiplication map to $\RLBSch{G}{x}$
and let \cdef{$\RLBSch{G_w}{x} \subset \RLBSch{G}{x}$} be the (locally closed) subscheme $\RLBSch{U_w}{x} \dot{w} \RLBSch{B}{x}$, where $\dot{w} \in \RLBSch{G}{x}(\Rq)$ is a representative for $w$. Then $\RLBSch{U_w}{x}$ is isomorphic to $\AA^{l(w)}_\sch{S}$ and $\RLBSch{G_w}{x}$ is isomorphic to $\AA^{l(w)}_\sch{S} \times \RLBSch{B}{x}$. Let \cdef{$w_0$} be the Coxeter element in $W_{x}$ (recall that $\Phi_{x}$ is a reduced root system). Then $\RLBSch{G_{w_0}}{x}\subset \RLBSch{G}{x}$ is an open subscheme and $\RLBSch{G}{x} = \mathop{\cup}\limits_{w\in W_{x}} \dot{w} \RLBSch{G_{w_0}}{x} \dot{w}^{-1}$ is an open covering.

We can now define $\RLBSch{G}{x}/\RLBSch{B}{x}$ by gluing data, as follows. For each $w\in W_x$, let $\RLBSch{b(w)}{x}: \dot{w}\RLBSch{G_{w_0}}{x} \dot{w}^{-1} \to \AA^{l(w_0)}_{\sch{S}}$ be the obvious map (conjugate to $\RLBSch{G_{w_0}}{x}$, then use $\RLBSch{G_{w_0}}{x}\iso \AA^{l(w_0)}_\sch{S} \times \RLBSch{B}{x'}$ and finally project to $\AA^{l(w_0)}_\sch{S}$). For each pair $w_1, w_2\in W_{x}$, let \cdef{$V_{w_1} = \AA^{l(w_0)}_\sch{S}$}; also, let \cdef{$V_{w_1,w_2}$} be the image of $\dot{w_1}\RLBSch{G_{w_0}}{x} \dot{w_1}^{-1} \cap \dot{w_2}\RLBSch{G_{w_0}}{x} \dot{w_2}^{-1}$ under $\RLBSch{b(w_1)}{x} : \dot{w_1}\RLBSch{G_{w_0}}{x} \dot{w_1}^{-1} \to \AA^{l(w_0)}_{\sch{S}}$. For each pair $w_1, w_2\in W_{x}$, 
glue $V_{w_1}$ to $V_{w_2}$ along $V_{w_1,w_2} \iso V_{w_2,w_1}$. The resulting scheme is \cdef{$\RLBSch{G}{x}/\RLBSch{B}{x}$}.

\[
\xymatrix{
\LBsch{B'} \ar[d]  \ar@{=}[r] & \LBSch{B}{x} \ar[r]\ar[d] & \RLBSch{B}{x} \ar[d] & \ar[l] \lBSch{B}{x} \ar[d]  \\
\LBsch{G}\ar[d]  \ar@{=}[r] & \LBSch{G}{x} \ar[r]\ar[d]^{\LBSch{b}{x}} & \RLBSch{G}{x} \ar[d]^{\RLBSch{b}{x}} & \ar[l] \lBSch{G}{x} \ar[d]^{\lBSch{b}{x}} \\
\LBsch{G}/\LBsch{B}  \ar@{=}[r] & \left(\RLBSch{G}{x}/\RLBSch{B}{x}\right)_{\eta} \ar[r] & \RLBSch{G}{x}/\RLBSch{B}{x} & \ar[l]  \left(\RLBSch{G}{x}/\RLBSch{B}{x}\right)_{{\bf s}}  \\
}
\]

The paragraph above defines $\RLBSch{G}{x}/\RLBSch{B}{x}$ and also \cdef{$\RLBSch{b}{x} : \RLBSch{G}{x}\to \RLBSch{G}{x}/\RLBSch{B}{x}$}. It is clear that $\RLBSch{b}{x}$ is a principal fibration with group 
$\RLBSch{B}{x}$. Since this fibration is given locally by $\RLBSch{b(w)}{x}$ --- which is defined by composing two isomorphisms and then projecting $\AA^{l(w_0)}_\sch{S} \times \RLBSch{B}{x}$ --- the smoothness of the fibration follows from Lemma~\ref{lemma: borel}.

Finally, the statements about the generic and the special fibre are clear from the 
construction of $\RLBSch{G}{x}/\RLBSch{B}{x}$.
\end{proof}



One more lemma is required for the proof of Theorem~\ref{theorem: induction}.
With notation as in Subsubsection~\ref{subsubsection: flag} and hypotheses as in Lemma~\ref{lemma: borel}, we define:
\begin{equation*}
\begin{aligned}
&\cdef{\RLBSch{X}{x} \ceq \left\{ (g,h) \in \RLBSch{G}{x}\times\RLBSch{G}{x} \tq h^{-1}gh\in \RLBSch{B}{x}\right\}}\\
&\cdef{\RLBSch{Y}{x} \ceq\left\{ (g,h\RLBSch{B}{x}) \in \RLBSch{G}{x}\times\left(\RLBSch{G}{x}/\RLBSch{B}{x}\right) \tq h^{-1}gh\in \RLBSch{B}{x}\right\}}
\end{aligned}
\end{equation*}
and $\cdef{\RLBSch{\beta}{x} : \RLBSch{X}{x} \to  \RLBSch{Y}{x}}$ by $\RLBSch{\beta}{x}(g,h) \ceq (g,h\RLBSch{B}{x})$.

\begin{lemma}\label{lemma: fibration}
The map $\RLBSch{\beta}{x} : \RLBSch{X}{x} \to \RLBSch{Y}{x}$ is a smooth principal 
$\RLBSch{B}{x}$-fibration. Moreover, the generic fibre $\KSch{\beta}{x}$ of $\RLBSch{\beta}{x}$ 
is the smooth principal fibration $\beta_\LBsch{B}$, the latter map defined in Subsection~\ref{subsection: induction}.
Moreover,
$\kSch{\beta}{x} = \beta_\kSch{B}{x}$, the latter map also defined in 
Subsection~\ref{subsection: induction}.
\end{lemma}

\begin{proof}
This is a direct consequence of the definitions above, the definitions of Subsection~\ref{subsection: induction} and Lemmas~\ref{lemma: borel} and \ref{lemma: flag}.
\end{proof}

\subsubsection{Proof of Theorem~\ref{theorem: induction}}\label{subsubsection: induced character sheaves}


\begin{proof}{(Proof of Theorem~\ref{theorem: induction})}
Since $(x)$ is a hyperspecial (poly)vertex by hypothesis, $\nu_{\RRSch{G}{x}} = \id$ and $\nu_{\RRSch{T}{x}} = \id$ (refer to Subsection~\ref{subsection: models}). 
%
%
Consider the integral model $\RLBSch{B}{x}$ for $\LBsch{B}$ introduced in Lemma~\ref{lemma: borel}. 
Observe that the reductive quotient map \cdef{$\LBsch{r} : \LBsch{B}\to \Lsch{T}$} is a $\Kq$-morphism. 
Let $\RLBSch{r}{x} : \RLBSch{B}{x} \to \RLBSch{T}{x}$ be the unique extension of $\LBsch{r}$ (existence and uniqueness is given by the Extension Principle of \cite[2.3]{Yu}.) Likewise define $\RLBSch{B}{x} \hookrightarrow \RLBSch{G}{x}$ with generic fibre $\LBsch{B} \hookrightarrow \Lsch{G}$.

Using notation from Subsubsection~\ref{subsubsection: flag}, 
define \cdef{$\RLBSch{\alpha}{x} : \RLBSch{X}{x} \to \RLBSch{T}{x}$} by 
$\RLBSch{\alpha}{x}(h,p) = \RLBSch{r}{x}(h^{-1}gh)$. We remark that $\RLBSch{\alpha}{x}$ 
is smooth.
%
Let \cdef{$\RLBSch{\gamma}{x} : \RLBSch{Y}{x} \to \RLBSch{G}{x}$} be projection onto the first component. 
We remark that $\RLBSch{\gamma}{x}$ is proper.
Recalling the description of parabolic induction of character sheaves from Subsection~\ref{subsection: induction}, it suffices to show that
\begin{eqnarray*}
\NC{\RRLBSch{G}{x}}\ {\gamma_\LLBsch{B}}_!\ (\beta_\LLBsch{B})_\# \ {\alpha_\LLBsch{B}}^*\ \fais{G}
&=&  {\gamma_{\llSch{B}{x}}}_!\ (\beta_{\llSch{B}{x}})_\# \  {\alpha_\llSch{B}{x}}^*\ 
\NC{\RRSch{T}{x}}\ \fais{G}
\end{eqnarray*}
The left-hand column of the diagram above gives
	\begin{equation*}
	\begin{aligned}
 & \NC{\RRLBSch{G}{x'}}\ {\gamma_\LLBsch{B'}}_!\ (\beta_\LLBsch{B'})_\# \ {\alpha_\LLBsch{B'}}^*\ \fais{G}\\
&= 
\NC{\RRLBSch{G}{x'}}\ {\LLBSch{\gamma}{x'}}_!\ (\LLBSch{\beta}{x'})_\#\ {\LLBSch{\alpha}{x'}}^*\ \fais{G}.
	\end{aligned}
	\end{equation*}
Since $\RLBSch{\gamma}{x'}$ is proper, proper base chance (\cf Subsection~\ref{item: FNC_!}) provides a natural isomorphism
	\begin{equation*}
	\begin{aligned}
&\NC{\RRLBSch{G}{x'}}\ {\LLBSch{\gamma}{x'}}_!\ (\LLBSch{\beta}{x'})_\#\ {\LLBSch{\alpha}{x'}}^*\ \fais{G}\\
&\iso 
{\llBSch{\gamma}{x'}}_!\ \NC{\RRLBSch{Y}{x'}}\ (\LLBSch{\beta}{x'})_\#\ {\LLBSch{\alpha}{x'}}^*\ \fais{G}.
	\end{aligned}
	\end{equation*}
Since $\lBSch{\gamma}{x'} = \gamma_{\lBSch{B'}{x'}}$, it follows that
	\begin{equation*}
	\begin{aligned}
&{\llBSch{\gamma}{x'}}_!\ \NC{\RRLBSch{Y}{x'}}\ (\LLBSch{\beta}{x'})_\#\ {\LLBSch{\alpha}{x'}}^*\ \fais{G}\\
&\iso {\gamma_{\llBSch{B'}{x'}}}_!\ 
\NC{\RRLBSch{Y}{x'}}\ (\LLBSch{\beta}{x'})_\#\ {\LLBSch{\alpha}{x'}}^*\ \fais{G}.
	\end{aligned}
	\end{equation*}
Lemma~\ref{lemma: fibration} ensures that the hypotheses of Proposition~\ref{proposition: NCe} are met, so Proposition~\ref{proposition: NCe} provides a natural isomorphism
	\begin{equation*}
	\begin{aligned}
& {\gamma_{\llBSch{B'}{x'}}}_!\ 
\NC{\RRLBSch{Y}{x'}}\ (\LLBSch{\beta}{x'})_\#\ {\LLBSch{\alpha}{x'}}^*\ \fais{G}\\
&\iso {\gamma_{\llBSch{B'}{x'}}}_!\ 
(\llBSch{\beta}{x'})_\#\ \NC{\RRLBSch{X}{x'}}\ {\LLBSch{\alpha}{x'}}^*\ \fais{G}.
	\end{aligned}
	\end{equation*}
In Lemma~\ref{lemma: fibration} we saw that $\lBSch{\beta}{x'} =  \beta_{\lBSch{B'}{x'}}$, so
	\begin{equation*}
	\begin{aligned}
& {\gamma_{\llBSch{B'}{x'}}}_!\ 
(\llBSch{\beta}{x'})_\#\ \NC{\RRLBSch{X}{x'}}\ {\LLBSch{\alpha}{x'}}^*\ \fais{G}\\
& = {\gamma_{\llBSch{B'}{x'}}}_!\ 
(\beta_{\llBSch{B'}{x'}})_\#\ \NC{\RRLBSch{X}{x'}}\ {\LLBSch{\alpha}{x'}}^*\ \fais{G}.
	\end{aligned}
	\end{equation*}
Since $\LBSch{\alpha}{x'}$ is smooth, Subsection~\ref{item: SNC^*} provides a natural isomorphism
	\begin{equation*}
	\begin{aligned}
& {\gamma_{\llBSch{B'}{x'}}}_!\ (\beta_{\llBSch{B'}{x'}})_\# 
\  \NC{\RRLBSch{X}{x'}}\ {\LLBSch{\alpha}{x'}}^*\ \fais{G} \\
& \iso {\gamma_{\llBSch{B'}{x'}}}_!\ (\beta_{\llBSch{B'}{x'}})_\# 
\  {\llBSch{\alpha}{x'}}^*\ \NC{\RRLBSch{T}{x'}}\  \fais{G}.
	\end{aligned}
	\end{equation*}
As observed above, $\llBSch{\alpha}{x'} =  \alpha_\llBSch{B'}{x'}$, so
	\begin{equation*}
	\begin{aligned}
& {\gamma_{\llBSch{B'}{x'}}}_!\ (\beta_{\llBSch{B'}{x'}})_\# 
\  {\llBSch{\alpha}{x'}}^*\ \NC{\RRLBSch{T}{x'}}\ \fais{G}\\
&= {\gamma_{\llBSch{B'}{x'}}}_!\ (\beta_{\llBSch{B'}{x'}})_\# \  {\alpha_\llBSch{B'}{x'}}^*\ 
 \NC{\RRLBSch{T}{x'}}\ \fais{G}.
	\end{aligned}
	\end{equation*}
This concludes the proof of all but the last sentence of Theorem~\ref{theorem: induction}.
To finish, we need only one more remark: $\RES{T}{x}  =\NC{\RRLBSch{T}{x}}$ and $\RES{G}{x}  =\NC{\RRLBSch{G}{x}}$ since $(x)$ is hyperspecial.
\end{proof}

\begin{corollary}\label{corollary: induction}
Let $\Ksch{T}$ be an unramified maximal torus in $\Ksch{G}$ and let $\Lq$ be the splitting extension in $\KKq$ for $\Ksch{T}$. Let $\LBsch{B'}$ be a Borel subgroup of $\Lsch{G}$ such that the Levi component of $\LLBsch{B'}=\LBsch{B'}\times_\Spec{\Lq} \Spec{\KKq}$ is $\LLBsch{T}$. Suppose $x$ lies in the image of $I(\Ksch{T},\Kq) \hookrightarrow I(\Ksch{G},\Kq)$ and let $x'$ be the image of $x$ under $I(\Ksch{G},\Kq) \hookrightarrow I(\Lsch{G},\Lq)$. Suppose $(x)$ is hyperspecial. Then there is a smooth integral model $\RLBSch{B'}{x'}$ for $\LBsch{B'}$ such that $\RLBSch{B'}{x'}(\RLq) = \LBsch{B'}(\Lq) \cap \RLBSch{G}{x'}(\RLq)$ and  the special fibre $\lBSch{B'}{x'}$ of $\RLBSch{B'}{x'}$ is a Borel subgroup of $\llBSch{G}{x'}$. Moreover,
\[
\NC{\RRSch{G}{x}}\ \ind^{\KKsch{G}}_{\LLBsch{B'}}\ \fais{G} \iso \ind^{\kkSch{G}{x}}_{\llSch{B'}{x'}}\ \NC{\RRSch{T}{x}}\ \fais{G},
\]
for every character sheaf $\fais{G}$ of $\KKsch{T}$.
Consequently,
\[
\RES{G}{x}\ \ind^{\KKsch{G}}_{\LLBsch{B'}}\ \fais{G} \iso \ind^{\kkSch{G}{x}}_{\llSch{B'}{x'}}\ \RES{T}{x}\ \fais{G}.
\]
\end{corollary}



\section{Matching sheaves with representations}\label{section: distributions}

In this section we explain how to `match' sheaves with representations. Then we provide relatively simple conditions from which one may conclude that a given sheaf matches a given representation, in this sense.
%

\subsection{Sheaves of Depth Zero}\label{subsection: Frobenius}

Suppose $x\in I(\Ksch{G},\Kq)$. Recall (from Subsection~\ref{subsection: models}) that $\rkSch{G}{x}$ is a reductive algebraic group over $\kq$; thus, $\rkkSch{G}{x}$ is defined over $\kq$, and so admits a Frobenius automorphism, henceforth denoted by $\Frob_{\rkSch{G}{x}} : \rkkSch{G}{x} \to \rkkSch{G}{x}$.

\begin{definition}\label{definition: depth zero}
An equivariant perverse sheaf $\fais{F}$ on $\KKsch{G}$ is said to have \cdef{depth zero} if
$\RESFAIS{F}{G}{x}$ is Frobenius-stable (in the sense of Subsection~\ref{subsection: characteristic function}) for each $x\in I(\Ksch{G},\Kq)$; \ie if there is an isomorphism  
\[
\Frob_{\rkSch{G}{x}}^* \RESFAIS{F}{G}{x} \iso \RESFAIS{F}{G}{x}.
\]
in $D^b_c(\rkkSch{G}{x},\EE)$, for each $x\in I(\Ksch{G},\Kq)$.
\end{definition}

Let $\fais{F}$ be an equivariant perverse sheaf of depth zero. Fix $x\in I(\Ksch{G},\Kq)$ and let 
\[
\varphi_{\fais{F},x} : \Frob_{\rkSch{G}{x}}^* \RESFAIS{F}{G}{x}\to\RESFAIS{F}{G}{x}
\] 
be an isomorphism in $D^b_c(\rkkSch{G}{x};\EE)$. Suppose $x\leq y$ in $I(\Ksch{G},\Kq)$. Since $\rkPsch{G}{x\leq y}$ is a parabolic subgroup of $\rkSch{G}{x}$ with Levi component $\rkSch{G}{y}$ (notice that these are all schemes over $\kq$, see Theorem~\ref{theorem: restriction}),  there is a canonical isomorphism
\begin{equation}\label{equation: r&f}
\res^{\rkkSch{G}{x}}_{\rkkPsch{G}{x\leq y}}\ \Frob_{\rkSch{G}{x}}^* \iso \Frob_{\rkSch{G}{y}}^*\ \res^{\rkkSch{G}{x}}_{\rkkPsch{G}{x\leq y}}
\end{equation}
Using this together with the full force of Theorem~\ref{theorem: restriction}, we define an isomorphism
\[
\varphi_{\fais{F},x\leq y} : \Frob_{\rkSch{G}{y}}^* \RESFAIS{F}{G}{y}\to \RESFAIS{F}{G}{y}
\] 
in $D^b_c(\rkkSch{G}{y};\EE)$ by the following diagramme.
\[
\xymatrix{
\res^{\rkkSch{G}{x}}_{\rkkPsch{G}{x\leq y}} \Frob_{\rkSch{G}{x}}^*\ \RESFAIS{F}{G}{x} \ar[d]_{\textrm{\eqref{equation: c&f}}}^{\iso} \ar[rr]^{\res^{\rkkSch{G}{x}}_{\rkkPsch{G}{x\leq y}}\ \varphi_{\fais{F},x}} && \res^{\rkkSch{G}{x}}_{\rkkPsch{G}{x\leq y}}\ \RESFAIS{F}{G}{x} \ar@{=}[d] \\
\Frob_{\rkSch{G}{y}}^*\ \res^{\rkkSch{G}{x}}_{\rkkPsch{G}{x\leq y}}\ \RESFAIS{F}{G}{x} \ar[d]_{\textrm{Theorem~\ref{theorem: restriction}}}^{\iso} &&\res^{\rkkSch{G}{x}}_{\rkkPsch{G}{x\leq y}}\ \RESFAIS{F}{G}{x} \ar[d]_{\textrm{Theorem~\ref{theorem: restriction}}}^{\iso} \\
\Frob_{\rkSch{G}{y}}^*\ \RESFAIS{F}{G}{y} \ar@{.>}[rr]^{\varphi_{\fais{F},x\leq y}} && \RESFAIS{F}{G}{y} 
}
\]
From Subsection~\ref{subsection: characteristic function} (see \eqref{equation: restriction and characteristic functions} in particular) it follows that
\begin{equation}\label{equation: tfs a}
\chf{\fais{F},x\leq y} =
\Res^{\rkSch{G}{x}(\kq)}_{\rkPsch{G}{x\leq y}(\kq)} \CHF{F}{G}{x}.
\end{equation}

Similarly, if $\fais{F}$ is an equivariant perverse sheaf of depth zero and $x\in I(\Ksch{G},\Kq)$ and $g\in \Ksch{G}(\Kq)$, then any isomorphism
\[
\varphi_{\fais{F},x} : \Frob_{\rkSch{G}{x}}^* \RESFAIS{F}{G}{x}\to\RESFAIS{F}{G}{x}
\] 
defines an isomorphism 
\[
\varphi_{\fais{F},g,x} : \Frob_{\rkSch{G}{gx}}^* \RESFAIS{F}{G}{gx}\to\RESFAIS{F}{G}{gx}
\]
in $D^b_c(\rkkSch{G}{gx};\EE)$ as follows. First, observe that 
\[
\Frob_{\rkSch{G}{x}}\circ \rkkSch{m(g^{-1})}{gx} =  \rkkSch{m(g^{-1})}{gx}\circ \Frob_{\rkSch{G}{gx}},
\] 
since $g\in \Ksch{G}(\Kq)$. Consequently, there is a canonical isomorphism
\begin{equation}\label{equation: c&f}
{\rkkSch{m(g^{-1})}{gx}}^*\ \Frob_{\rkSch{G}{x}}^* \iso \Frob_{\rkSch{G}{gx}}^*\ {\rkkSch{m(g^{-1})}{gx}}^*.
\end{equation}
Now the following diagramme defines $\varphi_{\fais{F},g,x}$.
\[
\xymatrix{
{\rkkSch{m(g^{-1})}{gx}}^*\ \Frob_{\rkSch{G}{x}}^*\ \RESFAIS{F}{G}{x} \ar[d]_{\textrm{\eqref{equation: c&f}}}^{\iso} \ar[rr]^{{\rkkSch{m(g^{-1})}{gx}}^*\ \varphi_{\fais{F},x}} && {\rkkSch{m(g^{-1})}{gx}}^*\ \RESFAIS{F}{G}{x} \ar@{=}[d] \\
\Frob_{\rkSch{G}{gx}}^*\ {\rkkSch{m(g^{-1})}{gx}}^*\ \RESFAIS{F}{G}{x} \ar[d]_{\textrm{Theorem~\ref{theorem: conjugation}}}^{\iso} &&{\rkkSch{m(g^{-1})}{gx}}^*\ \RESFAIS{F}{G}{x} \ar[d]_{\textrm{Theorem~\ref{theorem: conjugation}}}^{\iso} \\
\Frob_{\rkSch{G}{gx}}^*\ \RESFAIS{F}{G}{gx} \ar@{.>}[rr]^{\varphi_{\fais{F},g,x}}  && \RESFAIS{F}{G}{gx} 
}
\]

If $\fais{F}$ is an equivariant perverse sheaf of depth zero, then each isomorphism
\[
\varphi_{\fais{F},x} : \Frob_{\rkSch{G}{x}}^*\  \RES{G}{x} \fais{F}\to \RES{G}{x} \fais{F}
\]
determines a function \cdef{$\CHF{F}{G}{x} : \rkkSch{G}{x}(\kq) \to \EE$} as in Subsection~\ref{subsection: characteristic function}. 
By Subsection~\ref{subsection: characteristic function} (see \eqref{equation: pb and f} in particular), 
\begin{equation}\label{equation: tfs b}
\CHF{F}{G}{g,x} = \ \CHF{F}{G}{x} \circ \rkSch{m(g^{-1})}{gx}.
\end{equation}

\begin{definition}\label{definition: Frobenius structure}
A \emph{Frobenius structure} for an equivariant perverse sheaf of depth zero sheaf is a family \cdef{$\varphi_\fais{F} = (\varphi_{\fais{F},x})_{x\in I(\Ksch{G},\Kq)}$} of isomorphisms
\[
\varphi_{\fais{F},x} : \Frob_{\rkSch{G}{x}}^* \RESFAIS{F}{G}{x} \to \RESFAIS{F}{G}{x},
\]
that satisfies the following properties:
\begin{enumerate}
\item[(a)]
if $x\leq y$ in the Bruhat order for $I(\Ksch{G},\Kq)$ then 
\[
\CHF{F}{G}{y} = \Res^{\rkkSch{G}{x}(\kq)}_{\rkkPsch{G}{x\leq y}(\kq)} \CHF{F}{G}{x}  
\]
where $\rkkPsch{G}{x\leq y}$ is the parabolic subgroup of $\rkkSch{G}{x}$ appearing in Theorem~\ref{theorem: restriction};
\item[(b)]
if $x\in I(\Ksch{G},\Kq)$ and $g\in \Ksch{G}(\Kq)$ then 
\[
\CHF{F}{G}{gx} = \CHF{F}{G}{x} \circ \rkkSch{m(g^{-1})}{gx},
\]
where $\rkkSch{m(g^{-1})}{gx} : \rkkSch{G}{gx} \to \rkkSch{G}{x}$ is the isomorphism appearing in Theorem~\ref{theorem: conjugation}.
\end{enumerate}
\end{definition}


\subsection{Matching Character Sheaves with Representations}\label{subsection: matching}

We now prepare for the main definition of the paper. 

For each $x\in I(\Ksch{G},\Kq)$, let \cdef{$\rho_{x} : \Rsch{G}_x(\Rq) \to \rkSch{G}{x}(\kq)$} be the composition of: the group homomorphism 
\[
\Hom_{\Spec{\Rq}}(\Spec{\Rq},\Rsch{G}_x) \to \Hom_{\Spec{\Rq}}(\Spec{\kq},\Rsch{G}_x)
\]
defined by composition with the canonical map $\Spec{\kq} \to \Spec{\Rq}$; the identification
$\Rsch{G}_x(\kq) = \kSch{G}{x}(\kq)$; and the map of $\kq$-rational points $\kSch{G}{x}(\kq) \to \rkSch{G}{x}(\kq)$
induced from $\nu_{\RSch{G}{x}}$. We will refer to $\rho_{x}$ as a \cdef{reduction map}. Alternatively, the reduction map $\rho_{x}$ may be defined by
\[
\RSch{G}{x}(\Rq) = \Ksch{G}(\Kq)_x \to \Ksch{G}(\Kq)_x/\Ksch{G}(\Kq)_{x,0^+} = \rkSch{G}{x}(\kq).
\]
Observe that $\rho_x$ is a group homomorphism only; it is not a map of ringed spaces nor is it a map of points induced by a map of ringed spaces.

Let $\pi : \Ksch{G}(\Kq) \to \Aut_\EE(V)$ be an admissible $\ell$-adic representation. For each $x\in I(\Ksch{G},\Kq)$, let 
\[
\cdef{\cRes^{\Ksch{G}(\Kq)}_{\RSch{G}{x}(\Rq)}\pi : \rkSch{G}{x}(\kq) \to \Aut_\EE(V^{\Ksch{G}(\Kq)_{x,0^+}})}
\] 
be the representation of $\rkSch{G}{x}(\kq)$ obtained by factoring $\pi\vert_{\RSch{G}{x}(\Rq)}$ through $\rho_x : \RSch{G}{x}(\Rq) \to \rkSch{G}{x}(\kq)$ (this uses the fact $\rkSch{G}{x}(\kq) = \Ksch{G}(\Kq)_x/\Ksch{G}(\Kq)_{x,0^+}$). The representation \[(\cRes^{\Ksch{G}(\Kq)}_{\RSch{G}{x}(\Rq)}\pi, V^{\Ksch{G}(\Kq)_{x,0^+}})\] is called the \cdef{compact restriction of $\pi$ to $\rkSch{G}{x}(\kq)$}.

Consider the Grothendieck group \cdef{$R_\ZZ(\Ksch{G}(\Kq))$} of admissible $\ell$-adic representations of $\Ksch{G}(\Kq)$ and let \cdef{$R^0_\ZZ(\Ksch{G}(\Kq))$} be the subgroup generated by admissible $\ell$-adic representations of depth zero. 



\begin{definition}\label{definition: matching}
An equivariant perverse sheaf $\fais{F}$ on $\KKsch{G}$ with Frobenius structure 
$\varphi_\fais{F}=(\varphi_{\fais{F},x})_{x\in I(\Ksch{G},\Kq)}$ \cdef{matches} an element $\sum_m b_m [\pi_m]$ of $R^0_\ZZ(\Ksch{G}(\Kq))\otimes_\ZZ \QQ$ if
\[
\CHF{F}{G}{x} = \sum_m b_m\, \Trace \cRes^{\Ksch{G}(\Kq)}_{\RSch{G}{x}(\Rq)}\pi_m,
\]
for each $x\in I(\Ksch{G},\Kq)$.
\end{definition}

\begin{proposition}\label{proposition: matching}
Let $x_0,  x_1, \ldots , x_{N-1}$ be the vertices of a chamber in the Bruhat-Tits building $I(\Ksch{G},\Kq)$.
Suppose $\fais{F}$ is an equivariant perverse sheaf on $\KKsch{G}$ with Frobenius-structure $\varphi_\fais{F}$. Let $\sum_m b_m [\pi_m]$ be an element of $R^0_\ZZ(\Ksch{G}(\Kq))\otimes_\ZZ \QQ$. If
\[
\CHF{F}{G}{x_i} = \sum_m b_m\, \Trace \cRes^{\Ksch{G}(\Kq)}_{\RSch{G}{x_i}(\Rq)}\pi_m,
\]
for all $i\in \{ 0,1, \ldots , N-1\}$, then $\fais{F}$ matches $\sum_m b_m [\pi_m]$ in the sense of Definition~\ref{definition: matching}.
\end{proposition}

\begin{proof}
Without loss of generality, in this proof we replace $\sum_m b_m [\pi_m]$ with $[\pi]$.

First, suppose $x$ is any element of the chamber with vertices  $x_0, x_1, \ldots , x_{N-1}$. Then 
$x_i\leq x$ for some (poly)vertex $(x_i)$ of the chamber. Then, by \cite[C.1.3]{V2},
\begin{equation}\label{equation: SS 1}
\cRes^{\Ksch{G}(\Kq)}_{\RSch{G}{x}(\Rq)} \pi = \Res^{\rkSch{G}{x_i}(\kq)}_{\rkPsch{G}{x_i\leq x}(\kq)} \cRes^{\Ksch{G}(\Kq)}_{\RSch{G}{x_i}(\Rq)} \pi,
\end{equation}
where $\Res^{\rkSch{G}{x_i}(\kq)}_{\rkPsch{G}{x\leq x_i}(\kq)}$ is the restriction functor on representations. (See Subsection~\ref{subsection: rrq}, especially \eqref{equation: r}, for the definition of $\ksch{\tau}_{x\leq x_i} : \rkSch{G}{x\leq x_i} \to \rkSch{G}{x_i}$.)

and if $\ksch{\tau}_{x_i\leq x} : \rkSch{G}{x_i\leq x} \to \rkSch{G}{x_i}$ 
and $\ksch{\iota}_{x_i\leq x} : \rkSch{G}{x_i\leq x} \to \rkSch{G}{x}$
are defined as in Subsection~\ref{subsection: rrq},
then \eqref{equation: SS 1} implies that
for all $g\in \rkSch{G}{x}(\kq)$,
\begin{equation}\label{equation: SS 2}
\begin{aligned}
&\Trace\left(\cRes^{\Ksch{G}(\Kq)}_{\RSch{G}{x}(\Rq)} \pi \right)(g)\\
&= {q^{-\dim\ksch{\tau}_{x_i\leq x}}} 
\sum_{p\in \ksch{\tau}_{x_i\leq x}^{-1}(g)}\!\!\!\!\!\Trace\left(\cRes^{\Ksch{G}(\Kq)}_{\RSch{G}{x_i}(\Rq)} \pi\right)(\ksch{\iota}_{x_i\leq x}(p)).
\end{aligned}
\end{equation} 
It follows from Definition~\ref{definition: depth zero}, Part (a)
and the definition of parabolic restriction of functions that
\begin{equation}\label{equation: SS 5}
\CHF{F}{G}{x}(g) = {\Res^{\rkSch{G}{x_i}(\kq)}_{\rkSch{G}{x_i\leq x}(\kq)}\ \chi_{\varphi_{\fais{F},x_i}}}  =
{q^{-\dim\ksch{\tau}_{x_i\leq x}}}
\sum_{p\in \ksch{\tau}_{x_i\leq x}^{-1}(g)} \CHF{F}{G}{x_i}(\ksch{\iota}_{x_i\leq x}(p)),
\end{equation}
for all $g\in \rkSch{G}{x}(\kq)$.
Now, $\Trace\cRes^{\Ksch{G}(\Kq)}_{\RSch{G}{x_i}(\Rq)} \pi = \CHF{F}{G}{x_i}$ implies
\begin{equation}\label{equation: SS 6}
\CHF{F}{G}{x}(g) = {q^{-\dim\ksch{\tau}_{x_i\leq x}}}
\sum_{p\in \ksch{\tau}_{x_i\leq x}^{-1}(g)} \Trace\left(\cRes^{\Ksch{G}(\Kq)}_{\RSch{G}{x_i}(\Rq)}\pi \right) (\ksch{\iota}_{x_i\leq x}(p)),
\end{equation}
for all $g\in \rkSch{G}{x}(\kq)$.
Combining \eqref{equation: SS 6} with \eqref{equation: SS 2} it follows that
\begin{equation}\label{equation: SS 7}
\CHF{F}{G}{x} = \Trace\cRes^{\Ksch{G}(\Kq)}_{\RSch{G}{x}(\Rq)} \pi,
\end{equation}
for all $x$ in the chamber with vertices $x_0, x_1, \ldots , x_{N-1}$.

Next, suppose $y$ is an arbitrary element of $I(\Ksch{G},\Kq)$. Then there is some $g\in \Ksch{G}(\Kq)$ such that $y = gx$ for some $x$ in the chamber with vertices $x_0, x_1, \ldots , x_{N-1}$. Then $\left(\cRes^{\Ksch{G}(\Kq)}_{\RSch{G}{x}(\Rq)} \pi\right)\rkSch{m(g^{-1})}{y}$ and $\cRes^{\Ksch{G}(\Kq)}_{\RSch{G}{y}(\Rq)}\pi$ are equivalent representations (\cf Subsection~\ref{subsection: conjugation} for the definition of $\rkSch{m(g^{-1})}{y}$), and
\begin{equation}\label{equation: SS 8}
 \Trace\cRes^{\Ksch{G}(\Kq)}_{\RSch{G}{y}(\Rq)}\pi = \Trace\cRes^{\Ksch{G}(\Kq)}_{\RSch{G}{x}(\Rq)} \pi\ \circ \rkSch{m(g^{-1})}{y}.
\end{equation}
On the other hand, %
\begin{equation}\label{equation: SS 9}
\CHF{F}{G}{y} = \CHF{F}{G}{x} \circ \rkkSch{m(g^{-1})}{y}.
\end{equation}
But it was shown above that 
$\Trace\cRes^{\Ksch{G}(\Kq)}_{\RSch{G}{x}(\Rq)} = \CHF{F}{G}{x}$, from which it follows that
\begin{equation}\label{equation: SS 10}
 \Trace\cRes^{\Ksch{G}(\Kq)}_{\RSch{G}{y}(\Rq)}\pi =  \CHF{F}{G}{y},
\end{equation}
thus concluding the proof of Proposition~\ref{proposition: matching}.
\end{proof}

\subsection{Characteristic Distributions}\label{subsection: SS}

In this subsection we suppose that the connected centre of $\Ksch{G}(\Kq)$ is anisotropic. 

Let $\Hecke(\Ksch{G}(\Kq))$ be the Hecke algebra of locally constant functions $f : \Ksch{G}(\Kq)\to \EE$ supported by sets which are compact modulo the centre of $\Ksch{G}(\Kq)$. Let $\Hecke(\Ksch{G}(\Kq)^\elliptic)$ be the subspace of $\Hecke(\Ksch{G}(\Kq))$ consisting of functions supported by elliptic elements of $\Ksch{G}(\Kq)$.

\begin{definition}\label{definition: SS}
Suppose $\fais{F}$ is a character sheaf of $\KKsch{G}$ of depth zero with Frobenius-structure $\varphi_\fais{F}$.
The \cdef{characteristic distribution} of $\fais{F}$ (with respect to $\varphi_\fais{F}$) is the distribution $\Theta^0_{\fais{F}}$ on $\Hecke(\Ksch{G}(\Kq)^\elliptic)$ represented by the function
\[
\Theta^0_{\fais{F}}(g) \ceq \sum_{ (x) \subset I(\Ksch{G},\Kq) } (-1)^{\dim(x)}\ \CHF{F}{G}{x}\left(\rho_{x}(g)\right),
\]
where the union of taken over all facets $(x)$ in the Bruhat-Tits building of $\Ksch{G}(\Kq)$ such that $g\in \RSch{G}{x}(\Rq)$. 
\end{definition}

\begin{remark}
The sum in Definition~\ref{definition: SS} is taken over a set of facets of $I(\Ksch{G},\Kq)$. That set of facets is finite when $g$ is regular elliptic, as explained in \cite[4.9]{SS}.
Definition~\ref{definition: SS} is motivated by \cite[Prop.IV.1.5]{SS}, which will play an important role in the proof of Proposition~\ref{proposition: SS}.
\end{remark}

\begin{proposition}\label{proposition: SS}
Suppose the connected centre of $\Ksch{G}(\Kq)$ is anisotropic (and hence compact) and either $\Ksch{G}$ is split over $\Kq$ or $\Kq$ has characteristic $0$.
If $\fais{F}$ is a character sheaf of $\KKsch{G}$ of depth zero with Frobenius structure $\varphi_\fais{F}$ that matches the element $\sum_m b_m\, [\pi_m]$ of $R^0_\ZZ(\Ksch{G}(\Kq))\otimes_\ZZ \QQ$, then 
\[
\Theta^0_{\fais{F}}(f) = \sum_m b_m\, \Theta_{\pi_m}(f),
\] 
for all $f\in \Hecke(\Ksch{G}(\Kq)^\elliptic)$, where $\Theta_{\pi_m}$ is the character of the admissible representation $\pi_m$.
\end{proposition}

\begin{proof}
It will be enough to prove the following: if $\pi$ is a representation of $\Ksch{G}(\Kq)$ of depth zero and $\fais{F}$ is a perverse sheaf matching $\pi$, then 
\begin{equation}\label{equation: SS a}
\Theta^0_{\fais{F}}(g) = \Theta_{\pi}(g)
\end{equation}
for all regular elliptic $g\in \Ksch{G}(\Kq)$.
So, suppose $\pi$ matches $\fais{F}$.
Then
\begin{equation}\label{equation: SS b}
\CHF{F}{G}{x} =  \Trace \cRes^{\Ksch{G}(\Kq)}_{\RSch{G}{x}(\Rq)}\pi,
\end{equation}
for each $x\in I(\Ksch{G},\Kq)$.
Since $\pi$ is a representation of depth zero, it is admissible and finitely-generated. Thus, $\pi$ is of finite-length by \cite[3.12]{Ber}.
Since the connected centre of $\Ksch{G}(\Kq)$ is anisotropic and either $\Ksch{G}$ splits over $\Kq$ or $\Kq$ has characteristic $0$, it follows from \cite[III.4.10]{SS} and \cite[III.4.16]{SS} (with $e=0$) that
\begin{equation}\label{equation: SS c}
\Theta_\pi(g) = \sum_{k=0}^d (-1)^k \sum_{(x) \in I(\Ksch{G},\Kq)^g_k}\Trace\cRes^{\Ksch{G}(\Kq)}_{\RSch{G}{x}(\kq)}\pi(\rho_x(g)),
\end{equation}
where the inner sum is taken over over all $k$-dimensional facets of $I(\Ksch{G},\Kq)$ such that $g\in \RSch{G}{x}(\Rq)$, where $d$ is the maximal dimension of facets in $I(\Ksch{G},\Kq)$ and where $I(\Ksch{G},\Kq)^g_k$ of polyfacets of dimension $k$ whose stabilisers contain $g$. This set of polyfacets is finite since $g$ is regular elliptic; this is a consequence of \cite[Lem. III.4.9]{SS}.  Let $I(\Ksch{G},\Kq)^g$ be the union of the sets $I(\Ksch{G},\Kq)^g_k$ as $k$ ranges from $0$ to $d$. Then $I(\Ksch{G},\Kq)^g$ is finite and \eqref{equation: SS c} may be re-written in the form
\begin{equation}\label{equation: SS d}
\Theta_\pi(g) = \sum_{(x) \in I(\Ksch{G},\Kq)^g} (-1)^{\dim(x)} \Trace\cRes^{\Ksch{G}(\Kq)}_{\RSch{G}{x}(\kq)}\pi(\rho_x(g)).
\end{equation}
Combining this with \eqref{equation: SS b}, and recalling Definition~\ref{definition: SS}, we have
\begin{equation}\label{equation: SS e}
\Theta_\pi(g) = \sum_{(x) \in I(\Ksch{G},\Kq)^g} (-1)^{\dim(x)} \CHF{F}{G}{x}(\rho_x(g))
= \Theta^0_\fais{F}(g),
\end{equation}
thus proving \eqref{equation: SS a} and therefore concluding the proof of Proposition~\ref{proposition: SS}.
\end{proof}

\begin{remark}
The converse to Proposition~\ref{proposition: SS} is false.
\end{remark}

\section{Examples: Algebraic Tori and General Linear Groups}\label{section: examples}

This section is devoted to some simple examples of character sheaves of depth zero: in Subsection~\ref{subsection:  unramified tori} we show that unramified induced algebraic tori admit character sheaves of depth zero and in Subsection~\ref{subsection: induced} we show that $\KKsch{\GL(N)}$ also admits many character sheaves of depth zero (Proposition~\ref{proposition: induced}). The paper ends with Theorem~\ref{theorem: representations} in which we find sheaves that match many representations of $\GL(N,\Kq)$ of depth zero -- some of these representations are supercuspidal, some are not.



\subsection{Character Sheaves of Unramified Induced Tori}\label{subsection: unramified tori}

Let $\Ksch{T}$ be an induced (\cf \cite[4.1]{Yu}) algebraic torus over $\Kq$ which is
\emph{unramified}, \ie splits over an unramified extension. 
The Bruhat-Tits building $I(\Ksch{T},\Kq)$ is a single (poly)vertex $(x_0)$.

\begin{proposition}\label{proposition: unramified tori}
Let $\Ksch{T}$ be an unramified induced torus
and let $\theta : \Ksch{T}(\Kq) \to \EE^\times$  be an $\ell$-adic character of depth zero.
Then there is a character sheaf $\fais{F}$ of $\KKsch{T}$ of depth zero and a Frobenius-structure $\varphi_{\fais{F}}$ that matches $(-1)^{\dim\KKsch{T}}[\theta]$ in the sense of Definition~\ref{definition: matching}.
\end{proposition}

\begin{proof}
In this proof we use the notation of Subsection~\ref{subsection: local systems}; in particular, let $\psi : \mu_{\KKq} \to \EE^\times$ be a fixed injective $\ell$-adic character of the roots of unity in $\KKq$.
Let $p$ be the characteristic of the residue field $\kq$ and consider the group $\mu^p(\KKq)$ (resp. $\mu^p(\kkq)$) of roots of unity in $\KKq$ (resp. $\kkq$) with order prime to $p$. Hensel's Lemma provides a canonical isomorphism $\mu^p(\KKq)\iso \mu^p(\kkq)$. Let $\bar\psi : \mu^p(\kkq) \to \EE^\times$ be the character determined by $\psi$ and this isomorphism.

Let $\theta_{x_0}$ be the restriction of $\theta$ to $\RSch{T}{x_0}(\Rq)$;
let $\bar\theta_{x_0}$ be the compact restriction of $\theta$ to $\kSch{T}{x_0}(\kq)$. 
Let $\bar\lambda$ be a character of $\kkSch{T}{x_0}$ such that 
\begin{equation}
\bar\theta_{x_0}(t) = \bar\psi(\bar\lambda(t)),
\end{equation}
for all $t\in \kSch{T}{x_0}(\kq)$.

The torus $\Ksch{T}$ is unramified and $\RSch{T}{x_0}$ is the Neron model. We define a homomorphism $X(\KKsch{T}) \to X(\kkSch{T}{x_0})$ as follows. Let $\lambda$ be a character of $\KKsch{T}$. By the Extension Principle (\cf \cite[1.7]{BT2} or \cite[2.3]{Yu}), $\lambda$ extends uniquely to a morphism $\RRSch{\lambda}{x_0} : \RRSch{T}{x_0} \to \GL(1)_{\RKK}$. Restricting to special fibres produces $\kkSch{\lambda}{x_0}$, which is a character of  $\kkSch{T}{x_0}$. The homomorphism $X(\KKsch{T}) \to X(\kkSch{T}{x_0})$ thus defined is an isomorphism of character lattices.

Let $d$ be the order of $\theta_{x_0}$ (which is also the order of $\bar\theta_{x_0}$, and prime to $p$ since $\Ksch{T}$ is an induced torus) and consider the Kummer local system $\fais{L} \ceq \lambda^*\, \fais{E}_{d,\psi}$. Define $\fais{F} \ceq \fais{L}[\dim\KKsch{T}]$ and observe that $\fais{F}$ is a character sheaf of $\KKsch{T}$ (because character sheaves of algebraic tori are just Kummer local systems as sheaf complexes concentrated at the dimension of the torus, see Subsection~\ref{subsection: local systems}); in particular, $\fais{F}$ is a perverse sheaf on $\KKsch{T}$.

Since $\RRSch{\lambda}{x_0}$ is smooth and $\kkSch{\lambda}{x_0} = \bar\lambda$ and $\KKSch{\lambda}{x_0} = \lambda$, it follows from Subsection~\ref{item: SNC^*} that 
\[
\NC{\RRSch{T}{x_0}}\lambda^*\, \fais{E}_{d,\psi} \iso \bar\lambda^*\,  \NC{\RRSch{T}{x_0}}\fais{E}_{d,\psi}.
\]
A simple calculation shows that $ \NC{\RRSch{T}{x_0}}\fais{E}_{d,\psi}$ is the summand of the sheaf 
\[{\kkSch{[d]}{x_0}}_*\, (\EE)_{\kkSch{\GL(1)}{x_0}}\] on which $\mu_{d,\kkq}$ acts according to the character $\bar\psi$ (refer to Subsection~\ref{subsection: local systems}); in other words, 
\[
\NC{\RRSch{T}{x_0}}\fais{E}_{d,\psi} \iso \fais{E}_{d,\bar\psi}.
\]
(Since $\theta$ has depth zero, $d$ is invertible in $\kq$.)

Now, $\bar\lambda^*\,  \fais{E}_{d,\bar\psi}$ is Frobenius-stable (again, notice that $d$ is prime to $p$) and there is a \emph{canonical} isomorphism
\[
\chf{\varphi_{\bar\lambda^*\,  \fais{E}_{d,\bar\psi}}} : \Frob_{\kkSch{T}{x_0}}^* \bar\lambda^*\,  \fais{E}_{d,\bar\psi} \to \bar\lambda^*\,  \fais{E}_{d,\bar\psi}
\]
such that $\chf{\bar\lambda^*\,  \fais{E}_{d,\bar\psi}}(t) = \bar\psi(\bar\lambda(t))$, for all $t\in \kSch{T}{x_0}(\kq)$.
Since 
\[
\NC{\RRSch{T}{x_0}} \fais{L} = \bar\lambda^*\, \fais{E}_{d,\bar\psi},
\]
we define $\varphi_{\fais{F},x_0} \ceq \varphi_{\bar\lambda^*\,  \fais{E}_{d,\bar\psi}[\dim\KKsch{T}]}$. Since $I(\Ksch{T},\Kq) = \{ (x_0)\}$, this defines a Frobenius-structure for $\fais{L}$ that matches $(-1)^{\dim\KKsch{T}}[\theta]$, thus concluding the proof.
\end{proof}

\subsection{Some Character Sheaves of General Linear Groups}\label{subsection: induced}

In Subsection~\ref{subsection: unramified tori} we exhibited some examples of perverse sheaves with depth zero on unramified induced tori.
In this section we exhibit some examples of perverse sheaves with depth zero on general linear groups.
We write $\Ksch{G}$ for $\GL(N)_\Kq$ throughout Subsections~\ref{subsection: induced} and \ref{subsection: representations}.

\begin{proposition}\label{proposition: induced}
Let $\Ksch{T}$ be an unramified maximal torus in $\Ksch{G}$. Let $\theta$ be a character of $\Ksch{T}(\Kq)$ of depth zero and let $\fais{L}$ be the Kummer local system on $\KKsch{T}$ given in Proposition~\ref{proposition: unramified tori}. Let $\Lq$ be a splitting extension for $\Ksch{T}$ and let $\LBsch{B'}$ be a Borel subgroup of $\Lsch{G}\ceq\Ksch{G}\times_\Spec{\Kq} \Spec{\Lq}$ with Levi component $\Lsch{T}\ceq \Ksch{T}\times_\Spec{\Kq} \Spec{\Lq}$. Then 
$\ind_\KKsch{B'}^\KKsch{G} \fais{L}_\theta[\dim\KKsch{T}]$ 
is an equivariant perverse sheaf on $\KKsch{G}$ of depth zero. 
Moreover, $\RESFAIS{F}{G}{x}$ is a finite direct sum of character sheaves, for each $x\in I(\Ksch{G},\Kq)$.
\end{proposition}

\begin{proof}
Observe that $\dim\KKsch{T} = N$ and set $\fais{F} \ceq \ind_\KKsch{B'}^\KKsch{G} \fais{L}[N]$.
We show how to find $\RESFAIS{F}{G}{x}$ for each $x\in I(\Ksch{G},\Kq)$ and observe along the way that each $\RESFAIS{F}{G}{x}$ is a finite direct sum of character sheaves and also show how to define a Frobenius-structure for $\fais{F}$. 

We begin by finding $\RESFAIS{F}{G}{x}$ for $x$ in the standard (poly)vertex $(x_0)$. 
The building for $\Ksch{T}(\Kq)$ is naturally identified with the standard (poly)vertex $(x_0)$ in $I(\Ksch{G},\Kq)$. Let $x_0'$ be the image of $x_0$ in $I(\Lsch{G},\Lq)$. In this case, the special fibre $\kSch{G}{x_0}$ of $\Rsch{G}_{x_0}$ is reductive, so $\rkSch{G}{x_0} = \kSch{G}{x_0}$ and $\nu_{\RSch{G}{x_0}} = \id_{\kSch{G}{x_0}}$; likewise,  the special fibre $\kSch{T}{x_0}$ of $\RSch{T}{x_0}$ is reductive, so $\rkSch{T}{x_0} = \kSch{T}{x_0}$ and and $\nu_{\RSch{T}{x_0}} = \id_{\kSch{T}{x_0}}$. It now follows directly from Corollary~\ref{corollary: induction} that 
\[
\NC{\RRSch{G}{x_0}} \ind_{\KKsch{B'}}^{\KKsch{G}} \fais{L}[N] \iso \ind_{\llBSch{B'}{x_0'}}^{\kkSch{G}{x_0}} \NC{\RRSch{T}{x_0}} \fais{L}[N].
\]
Proposition~\ref{proposition: unramified tori} shows that $\NC{\RRSch{T}{x_0}} \fais{L}[N]$ is a Frobenius-stable character sheaf of $\kkSch{T}{x_0}$ with characteristic function equal to the virtual character $(-1)^N\bar\theta$ of $\kSch{T}{x_0}(\kq)$. To simplify notation slightly, we set $\fais{L}_{\bar\theta} \ceq \NC{\RRSch{T}{x_0}} \fais{L}$ below; thus, 
\begin{equation}\label{equation: induced 1}
\RESFAIS{F}{G}{x_0} = \ind^{\kkSch{G}{x_0}}_{\llBSch{B'}{x_0'}}\NC{\RRSch{T}{x_0}} \fais{L}[N]
= \ind^{\kkSch{G}{x_0}}_{\llBSch{B'}{x_0'}} \fais{L}_{\bar\theta}[N].
\end{equation}
Since $\fais{L}_{\bar\theta}[N]$ is a Frobenius-stable character sheaf of $\kkSch{T}{x_0}$ it 
follows from Subsection~\ref{subsection: characteristic function} that $\ind^{\kkSch{G}{x_0}}_{\llBSch{B'}{x_0'}} \fais{L}_{\bar\theta}[N]$ is also Frobenius-stable and equipped with a canonical isomorphism $\varphi_{\ind^{\kkSch{G}{x_0}}_{\llBSch{B'}{x_0'}} \fais{L}_{\bar\theta}[N]}$ from $\Frob_{\kkSch{G}{x_0}}^*\ \ind^{\kkSch{G}{x_0}}_{\llBSch{B'}{x_0'}} \fais{L}_{\bar\theta}[N]$ to $\ind^{\kkSch{G}{x_0}}_{\llBSch{B'}{x_0'}} \fais{L}_{\bar\theta}[N]$. We define 
\begin{equation}\label{equation: FS x0}
\varphi_{\fais{F},x_0} \ceq \varphi_{\ind^{\kkSch{G}{x_0}}_{\llBSch{B'}{x_0'}} \fais{L}_{\bar\theta}[N]}
\end{equation}
(\cf Subsection~\ref{subsection: characteristic function}).
We also remark that $\RESFAIS{F}{G}{x_0} = \ind^{\kkSch{G}{x_0}}_{\llBSch{B'}{x_0'}} \fais{L}_{\bar\theta}[N]$ is a finite direct sum of character sheaves of $\kkSch{T}{x_0}$ because $\fais{L}_{\bar\theta}[N]$ is a Frobenius-stable character sheaf of $\kkSch{T}{x_0}$ (by \cite[4.8(b)]{CS}, see also \cite[Cor.9.3.3]{MS}).

Next, let $x$ be any element of the Bruhat-Tits building of $\Ksch{G}(\Kq)$ which 
contains $x_0$ in its closure,  so $x_0\leq x$. By Theorem~\ref{theorem: restriction} and the paragraph above, 
\begin{equation}\label{equation: induced 3}
\RESFAIS{F}{G}{x} \iso   \res^{\kkSch{G}{x_0}}_{\kkSch{G}{x_0\leq x}} \  \ind^{\kkSch{G}{x_0}}_{\llBSch{B'}{x_0'}}\ \NC{\RRSch{T}{x_0}} \fais{L}[N].
\end{equation}
We saw above that $\RESFAIS{F}{G}{x_0} = \ind^{\kkSch{G}{x_0}}_{\llBSch{B'}{x_0'}}\ \NC{\RRSch{T}{x_0}} \fais{L}_\theta[N]$ is Frobenius-stable, with a canonical isomorphism (required to define its characteristic function) given by \eqref{equation: FS x0}.
By Subsection~\ref{subsection: characteristic function}, $\RESFAIS{F}{G}{x} = \res^{\kkSch{G}{x_0}}_{\kkSch{G}{x_0\leq x}} \ \RESFAIS{F}{G}{x_0}$ is also Frobenius-stable, with a canonical isomorphism given by
\begin{equation}\label{equation: FS x}
\varphi_{\fais{F},x} \ceq \varphi_{\res^{\kkSch{G}{x_0}}_{\kkSch{G}{x_0\leq x}} \RESFAIS{F}{G}{x}}
\end{equation}
(\cf Subsection~\ref{subsection: characteristic function}).
Moreover, since $\ind^{\kkSch{G}{x_0}}_{\llBSch{B'}{x_0'}}\ \NC{\RRSch{T}{x_0}} \fais{L}_\theta[N]$ is a finite sum of character sheaves, so is \[\res^{\kkSch{G}{x_0}}_{\kkSch{G}{x_0\leq x}} \ind^{\kkSch{G}{x_0}}_{\llBSch{B'}{x_0'}}\ \NC{\RRSch{T}{x_0}} \fais{L}_\theta[N],\] by \cite[6.9]{CS} (\cf also \cite[9.3.2(ii)]{MS}). 

Finally, let $y$ be an arbitrary element of Bruhat-Tits building of $\Ksch{G}(\Kq)$. 
Then there is some $x$ such that $x_0 \leq x$ and some $g\in \Ksch{G}(\Kq)$ such that $y = g x$ 
(we have just used a property of $\GL(N,\Kq)$!). By Theorem~\ref{theorem: conjugation} and \eqref{equation: induced 1},
\begin{equation}\label{equation: induced 2}
\RESFAIS{F}{G}{g x} \iso  {\kkSch{m(g^{-1})}{gx}}^*\ \res^{\kkSch{G}{x_0}}_{\kkSch{G}{x_0\leq x}} \  \ind^{\kkSch{G}{x_0}}_{\llBSch{B'}{x_0'}}\ \NC{\RRSch{T}{x_0}} \fais{L}[N].
\end{equation}
Moreover, since $\kkSch{m(g^{-1})}{gx}$ is defined over $\kq$ and $\ind^{\kkSch{G}{x_0}}_{\llBSch{B'}{x_0'}}\NC{\RRSch{T}{x_0}} \fais{L}[N]$ is Frobenius-stable (by work above), it follows that 
\[
{\kkSch{m(g^{-1})}{gx}}^*\ \ind^{\kkSch{G}{x_0}}_{\llBSch{B'}{x_0'}}\NC{\RRSch{T}{x_0}} \fais{L}[N]
\] 
is also Frobenius-stable. We define
\begin{equation}\label{equation: FS gx}
\varphi_{\fais{F},gx} \ceq \varphi_{\kkSch{m(g^{-1})}{gx}^*\ \RESFAIS{F}{G}{x}}
\end{equation}
(\cf Subsection~\ref{subsection: characteristic function}).
Moreover, $\ind^{\kkSch{G}{x_0}}_{\llBSch{B'}{x_0'}}\NC{\RRSch{T}{x_0}} \fais{L}[N]$ is a finite sum of character sheaves of $\kkSch{T}{x_0}$ (by work above) and since the functor ${\kkSch{m(g^{-1})}{y}}^*$ is exact, it follows that 
\[
{\kkSch{m(g^{-1})}{y}}^*\ \ind^{\kkSch{G}{x_0}}_{\llBSch{B'}{x_0'}}\NC{\RRSch{T}{x_0}} \fais{L}[N]
\] 
is also a finite direct sum of character sheaves of $\kkSch{T}{y}$. 

It only remains to show that $(\varphi_{\fais{F},y})_{y\in I(\Ksch{G},\Kq)}$ is a Frobenius-structure (\cf Definition~\ref{definition: depth zero}). Part (a)
of Definition~\ref{definition: depth zero} is clearly satisfied (use the transitivity of parabolic restriction). 
As for Part (b)
of Definition~\ref{definition: depth zero}, it boils down to showing that
$\varphi_{\fais{F},gx}$ is well-defined by \eqref{equation: FS gx}. We must show that if $g_1x = g_2x$ then $\chf{\varphi_{\fais{F},g_1,x}} = \chf{\varphi_{\fais{F},g_2,x}}$ (here we use notation from the discussion preceding Definition~\ref{definition: Frobenius structure}).
It suffices to show that if $g_0\in \KSch{G}{x}(\Rq)$ then $\chf{\varphi_{\fais{F},g_0,x}} = \chf{\varphi_{\fais{F},x}}$.
Since
\[
\chf{\varphi_{\fais{F},g_0,x}} = \chf{\varphi_{\fais{F},x}}\circ \rkSch{m(g_0^{-1})}{g_0x}
\]
and since $g_0x =x$ and $\RESFAIS{F}{G}{x}$ is equivariant, it follows that
\[
\chf{\varphi_{\fais{F},x}}\circ \rkSch{m(g_0^{-1})}{g_0x}
= \chf{\varphi_{\fais{F},x}},
\] 
as desired.
This concludes the proof of Proposition~\ref{proposition: induced}.
\end{proof}

\begin{remark}
In the context of Proposition~\ref{proposition: induced}, Mackey's formula for character sheaves (\cite[Prop.15.2]{CS} and \cite[Prop.10.1.2]{MS}), together with \eqref{equation: induced 3}, makes it easy to determine $\RESFAIS{F}{G}{x}$ for each $x\in I(\Ksch{G},\Kq)$.
\end{remark}


\subsection{From Representations to Character Sheaves}\label{subsection: representations}

As above, let $\Ksch{G}$ be the algebraic group $\GL(N)_\Kq$. In this section we consider generalised principal series representations of $\GL(N,\Kq)$, by which we mean representations of the form $\Ind^{\Ksch{G}(\Kq)}_{\Ksch{P}(\Kq)} \sigma$ where $\Ksch{P}$ is a parabolic subgroup of $\Ksch{G}$ and $\sigma$ is a supercuspidal representation of the Levi component $\Ksch{L}$ of $\Ksch{P}$. We include the possibility that $\Ksch{L} = \Ksch{P} = \Ksch{G}$ (\ie trivial induction). 

\begin{theorem}\label{theorem: representations}
Let $\pi$ be a generalised principal series representation of $\GL(N,\Kq)$ of depth zero. Then there is a perverse sheaf $\fais{F}$ on $\GL(N)_\KKq$ such that $\fais{F}$ matches $(-1)^N[\pi]$ in the sense of Definition~\ref{definition: matching}.
\end{theorem}

\begin{proof}
By definition, there is a parabolic subgroup $\Ksch{P}$ of $\Ksch{G}$ (not necessarily a proper subgroup) and a depth zero 
supercuspidal representation $\sigma$ of the Levi component $\Ksch{L}$ of $\Ksch{P}$ such that $\pi = \Ind^{\Ksch{G}(\Kq)}_{\Ksch{P}(\Kq)} \sigma$. 

Let $\Ksch{T}$ be an unramified maximal torus of $\Ksch{G}$ such that $\Ksch{T}\subseteq \Ksch{L}$ and $\Ksch{T}$ is an elliptic maximal torus of $\Ksch{L}$. (Up to conjugation, such a torus is unique.) Let $\Lq$ be a splitting extension for $\Ksch{T}$. Following our convention, we write $\Lsch{T}$ for $\Ksch{T}\times_\Spec{\Kq} \Spec{\Lq}$ and write $\Lsch{L}$ for $\Ksch{L}\times_\Spec{\Kq} \Spec{\Lq}$. Let $\LBsch{Q'}$ be a Borel subgroup of $\Lsch{L}$ with Levi component $\Lsch{T}$ and let $\LBsch{B'}$ be a Borel subgroup of $\Lsch{G}$ such that $\LBsch{B'}\subset \LBsch{P}$ and $\LBsch{B'}$ has Levi component $\Lsch{T}$.  If $\Ksch{T}$ is elliptic in $\Ksch{G}$ then $\Ksch{L} = \Ksch{G}$ and $\LBsch{P}=\Ksch{G}$ and $\LBsch{Q'} = \LBsch{B'}$; at the other extreme, if  $\Ksch{T}$ is split then $\Lq=\Kq$ and $\Ksch{L} = \Ksch{T}$ and $\LBsch{P}=\LBsch{B'}$ and $\LBsch{Q'} = \Ksch{T}$. In all cases, the square in the diagram below is 
commutative. As usual, we write $\Lsch{P}$ for $\Ksch{P}\times_\Spec{\Kq} \Spec{\Lq}$.
\begin{equation}
	\xymatrix{
&& \ar@{->>}[dl] \LBsch{B'} \ar@{>->}[dr] && \\
& \ar@{->>}[dl] \LBsch{Q'} \ar@{>->}[dr] && \ar@{->>}[dl] \Lsch{P} \ar@{>->}[dr] & \\
\Lsch{T} && \Lsch{L} && \Lsch{G} 
	}
\end{equation}
We may also assume $\Ksch{L}$ is in standard position, in which case $\Ksch{L} = \prod_{i=1}^n \Ksch{L_i}$ where each $\Ksch{L_i}$ is just $\GL(N_i)_\Kq$. Then $\Ksch{T} = \prod_{i=1}^n \Ksch{T_i}$ where $\Ksch{T_i}$ is an elliptic maximal unramified torus of $\Ksch{L_i}$; let $\Kq_i$ be a splitting field for $\Ksch{T_i}$ and let $N_i = [\Kq_i: \Kq]$; then $\sum_{i=1}^n N_i =N$.
Likewise, $\LBsch{Q'}$ admits the decomposition $\LBsch{Q'} = \prod_{i=1}^n \LBsch{Q'_i}$, where $\LBsch{Q'_i}$ is a Borel subgroup of $\Lsch{L_i}$ with Levi component $\Lsch{T_i}$.

Now, the building $I(\Ksch{T},\Kq)$ embeds canonically into the building $I(\Ksch{L},\Kq)$, which embeds canonically into the building $I(\Ksch{G},\Kq)$. Fix a (poly)vertex $x_0\in I(\Ksch{T},\Kq)$; we will identify $x_0$ with its image in $I(\Ksch{L},\Kq)$ and $I(\Ksch{G},\Kq)$, and denote its image in
$I(\Ksch{G}_\Lq,\Lq)$ by $x_0'$. Since we are working with general linear groups, all vertices in the buildings are hyperspecial.
Set $x_0 = (x_1, \ldots , x_n)$, so $x_i\in I(\Ksch{T_i},\Kq)$ lies is a  
vertex in $I(\Ksch{L_i},\Kq)$; likewise, set $x_0' = (x_1', \ldots , x_n')$, so $x_i'\in I(\Lsch{T_i},\Lq)$ lies in a (poly)vertex in $I(\Lsch{L_i},\Lq)$. Let $\RLBSch{Q'_i}{x_i'}$ be the schematic closure of $\LBsch{Q'_i}$ in $\RLBSch{L_i}{x_i'}$ (studied in Subsubsection~\ref{subsubsection: flag}); then $\RLBSch{Q'}{x_0'} = \prod_{i=1}^n \RLBSch{Q'_i}{x_i'}$ is the schematic closure of $\LBsch{Q'}$ in $\RLBSch{L}{x_0'}$.

Without loss of generality, we assume $\sigma$ is irreducible. 
Then $\sigma = \mathop{\otimes}_{i=1}^n \sigma_i$, where $\sigma_i$ is a supercuspidal irreducible representation of $\Ksch{L_i}(\Kq) = \GL(N_i,\Kq)$ of depth zero.

The proof of Theorem~\ref{theorem: representations} has three parts: first, for each $i$, we make an equivariant perverse sheaf $\fais{G}_i$ of depth zero on $\KKsch{L}_i$ which matches $\sigma_i$; then, we make an equivariant perverse sheaf $\fais{G}$ of depth zero on $\KKsch{L}$ which matches $\sigma$; finally, we make an equivariant perverse sheaf $\fais{F}$ of depth zero which matches $\pi$.

Fix $i$  and suppose $N_i > 1$. Consider the representation 
\[
\bar\sigma_i \ceq \cRes^{\Ksch{L_i}(\Kq)}_{\RSch{L_i}{x_i}(\Rq)} \sigma_i.
\]
Since $\sigma_i$ is a supercuspidal, irreducible representation of depth zero, $\bar\sigma_i$ is cuspidal and irreducible. 
Then $\kSch{T_i}{x_i}$ is an unramified anisotropic torus in $\kSch{L_i}{x_i}$ (unique up to conjugacy).

From \cite[Cor. 2.3]{shoji}, \cite[Rem. 1.12(ii)]{shoji}, 
and the description of cuspidal representations of the general linear
group over a finite field in the remark after \cite[Thm. 8.8]{srinivasan}, it follows that
for each cuspidal irreducible representation $\bar\sigma_i$ of $\kSch{L_i}{x_i}(\kq) = \GL(N_i,\kq)$ there is a character $\bar\theta_i : \kSch{T_i}{x_i}(\kq) \to \EE$ and a local system $\fais{L}_{\bar\theta_i}$ such that 
\begin{equation}\label{equation: representations 1}
\Trace\bar\sigma_i = \chf{\ind^{\kkSch{L_i}{x_i}}_{\llBSch{Q'_i}{x_i'}}\fais{L}_{\bar\theta_i}}.
\end{equation}

Let $\theta_i$ be a character of $\Ksch{T}(\Kq)$ such that $\bar\theta_i = \cRes^{\Ksch{T_i}(\Kq)}_{\RSch{T_i}{x_i}(\Rq)}\theta_i$. Using Proposition~\ref{proposition: unramified tori}, let $\fais{L}_{\theta_i}$ be a Kummer local system on $\KKsch{T_i}$ that matches the character $\theta_i : \Ksch{T_i}(\Kq) \to \EE^\times$. 

Define \[\fais{G}_i \ceq \ind^{\KKsch{L_i}}_{\LLBsch{Q'_i}} \fais{L}_{\theta_i}[N_i].\] 
Proposition~\ref{proposition: induced} shows that the equivariant perverse sheaf $\fais{G}_i$ has depth zero by showing how to define a Frobenius-structure $\varphi_{\fais{G}_i}$ for $\fais{G}_i$. 
We claim that $\fais{G}_i$ matches $\sigma_i$. 

Begin by using Corollary~\ref{corollary: induction} and Proposition~\ref{proposition: unramified tori} (and the definition of $\fais{G}_i$) to see that
\begin{eqnarray*}
\NC{\RRSch{L_i}{x_i}}\ \fais{G}_i 
&=& \NC{\RRSch{L_i}{x_i}}\ \ind^{\KKsch{L_i}}_{\LLBsch{Q'_i}} \fais{L}_{\theta_i}[N_i]  \\
&\iso&  \ind^{\kkSch{L_i}{x_i}}_{\llSch{Q'_i}{x_i'}} \NC{\RRSch{T_i}{x_i}}\ \fais{L}_{\theta_i}[N_i]  \\
&=&  \ind^{\kkSch{L_i}{x_i}}_{\llSch{Q'_i}{x_i'}} \fais{L}_{\bar\theta_i}[N_i] .
\end{eqnarray*}
Thus, $\NC{\RRSch{L_i}{x_i}}\ \fais{G}_i$ is Frobenius-stable (by Subsection~\ref{subsection: characteristic function} and the fact that $\fais{L}_{\bar\theta_i}[N_i]$ is Frobenius-stable), and \eqref{equation: representations 1} implies that
\begin{eqnarray}\label{equation: representations 2}
\chf{\fais{G}_i,x_i}
&= \chf{\ind^{\kkSch{L_i}{x_i}}_{\llSch{Q'_i}{x_i'}} \fais{L}_{\bar\theta_i}[N_i]}
= (-1)^{N_i} \Trace\bar\sigma_i.
\end{eqnarray}
Since $\Ksch{L_i}(\Kq) = \GL(N_i,\Kq)$, every (poly)vertex in $I(\Ksch{L_i},\Kq)$ is conjugate to the one containing $x_i$,
and it follows from Proposition~\ref{proposition: matching} and \eqref{equation: representations 2} that $\fais{G}_i$ matches the virtual representation $(-1)^{N_i}[\sigma_i]$. 

The case $N_i=1$ is simpler, and really just a degenerate case of the paragraph above: when $N_i =1$, $\Kq_i = \Kq$, and $\Ksch{L_i} = \Ksch{T_i}$, so $\sigma_i = \theta_i$ and $\fais{G}_i \ceq \fais{L}_{\theta_i}[1]$. This completes the first step in the proof of Theorem~\ref{theorem: representations}. 

For the second step in the proof, simply define
\begin{equation}
\fais{G} = \mathop{\boxtimes}_{i=1}^n \fais{G}_i.
\end{equation}
Observe that this is a character sheaf of $\KKsch{L}$ (\cf \cite[5.4.1]{MS} for example).
Recall that $x_0 = (x_1, \ldots, x_n)$. Then
\begin{equation}
\NC{\RRSch{L}{x_0}}\ \fais{G} = \mathop{\boxtimes}_{i=1}^n \NC{\RRSch{L}{x_i}} \fais{G}_i,
\end{equation}
and, from the Subsection~\ref{subsection: characteristic function} we have
\begin{equation}
\chf{\fais{G},x_0} = \mathop{\otimes}_{i=1}^n \chf{\fais{G}_i,x_i}.
\end{equation}
Arguing as above, it follows from Proposition~\ref{proposition: matching} that $\fais{G}$ matches $(-1)^N[\sigma]$.

For the third step in the proof, recall that $\Ksch{P}$ is a parabolic subgroup of $\Ksch{G}$ with Levi component $\Ksch{L}$ and define
\begin{equation}
\fais{F} = \ind^{\KKsch{G}}_{\KKsch{P}} \fais{G}.
\end{equation}
By the transitivity of parabolic induction (\cite[Prop.4.2]{CS}), 
\begin{equation}
\fais{F} = \ind^{\KKsch{G}}_{\KKsch{B'}} \fais{L}_\theta[N].
\end{equation}
It only remains to show that $\fais{F}$ matches the virtual representation $(-1)^N[\pi]$. 
Again, using Proposition~\ref{proposition: matching}, we see that it is sufficient to show
\begin{equation}\label{equation: representations 4}
(-1)^N\CHF{F}{G}{x_0} = \Trace\cRes^{\Ksch{G}(\Kq)}_{\RSch{G}{x_0}(\Rq)} \pi.
\end{equation}

To that end, consider the following diagram of integral schemes, 
\begin{equation}
	\xymatrix{
&& \ar@{->>}[dl] \RRSch{B'}{x_0'} \ar@{>->}[dr] && \\
& \ar@{->>}[dl] \RRSch{Q'}{x_0'} \ar@{>->}[dr] && \ar@{->>}[dl] \RRSch{P}{x_0'} \ar@{>->}[dr] & \\
\RRSch{T}{x_0'} && \RRSch{L}{x_0'} && \RRSch{G}{x_0'} 
	}
\end{equation}
in which $\RRSch{P}{x_0'}$ (resp. $\RRSch{Q'}{x_0'}$) denotes the integral closure of $\KKsch{P}$ in $\RRSch{G}{x_0'}$ (resp. $\KKsch{Q'}$ in $\RRSch{G}{x_0'}$).

Now, back to the left-hand side of \eqref{equation: representations 4}. Consider 
\[
\CHF{F}{G}{x_0} =  \chf{\RESFAIS{F}{G}{x_0}} =  \chf{\NC{\RRSch{G}{x_0}} \fais{F}}.
\]
Recall that $\fais{F} = \ind^{\KKsch{G}}_{\KKsch{B'}} \fais{L}_\theta[N]$. Define $\fais{L}_{\bar\theta} \ceq \mathop{\boxtimes}_{i=1}^n \fais{L}_{\bar\theta_i}$. By Corollary~\ref{corollary: induction}, 
\[
\NC{\RRSch{G}{x_0}} \fais{F} = \ind^{\kkSch{G}{x_0}}_{\llBSch{B'}{x_0'}} \NC{\RRSch{T}{x_0}}  \fais{L}_\theta[N].
\]
Thus,
\[
\chf{\NC{\RRSch{G}{x_0}} \fais{F}} 
= \chf{\ind^{\kkSch{G}{x_0}}_{\llBSch{B'}{x_0'}} \NC{\RRSch{T}{x_0}}  \fais{L}_\theta[N]} 
= \chf{\ind^{\kkSch{G}{x_0}}_{\llBSch{B'}{x_0'}} \NC{\RRSch{T}{x_0}}  \fais{L}_\theta[N]} .
\]
 Then, by construction, 
\[
\NC{\RRSch{T}{x_0}}  \fais{L}_\theta[N] = \fais{L}_{\bar\theta} [N].
\]
Thus,
\[
\chf{\ind^{\kkSch{G}{x_0}}_{\llBSch{B'}{x_0'}} \NC{\RRSch{T}{x_0}}  \fais{L}} 
 = \chf{\ind^{\kkSch{G}{x_0}}_{\llBSch{B'}{x_0'}} \fais{L}_{\bar\theta}} 
\]
By the paragraph above and the transitivity of parabolic induction (\cite[Prop.4.2]{CS}),
\[
\ind^{\kkSch{G}{x_0}}_{\llBSch{B'}{x_0'}} \fais{L}_{\bar\theta}[N]
\iso \ind^{\kkSch{G}{x_0}}_{\kkSch{P}{x_0}} \ind^{\kkSch{L}{x_0}}_{\kkSch{Q'}{x_0}} \fais{L}_{\bar\theta}[N].
\]
Thus, 
\[
\chf{\ind^{\kkSch{G}{x_0}}_{\llBSch{B'}{x_0'}} \fais{L}_{\bar\theta}} = \chf{\ind^{\kkSch{G}{x_0}}_{\kkSch{P}{x_0}} \ind^{\kkSch{L}{x_0}}_{\kkSch{Q'}{x_0}} \fais{L}_{\bar\theta}}.
\]
Referring back the the diagrams above, notice that $\kSch{P}{x_0}$ is a parabolic subgroup of $\kkSch{G}{x_0}$ over $\kq$; thus, 
\[
\chf{\ind^{\kkSch{G}{x_0}}_{\kkSch{P}{x_0}} \ind^{\kkSch{L}{x_0}}_{\kkSch{Q'}{x_0}} \fais{L}_{\bar\theta}}
= \Ind^{\kSch{G}{x_0}(\kq)}_{\kSch{P}{x_0}(\kq)} \chf{\ind^{\kkSch{L}{x_0}}_{\kkSch{Q'}{x_0}} \fais{L}_{\bar\theta}}.
\]
(This follows from basic properties of the {\it dictionnaire fonctions-faisceaux}, as remarked during the proof of \cite[Prop.15.2]{CS}.) Now, referring to the definition of $\fais{G}$ above, it follows from Theorem~\ref{theorem: induction} that
\[
\ind^{\kkSch{L}{x_0}}_{\kkSch{Q'}{x_0}} \fais{L}_{\bar\theta} = \NC{\RRSch{L}{x_0}} \fais{G}.
\]
Thus,
\[
\Ind^{\kSch{G}{x_0}(\kq)}_{\kSch{P}{x_0}(\kq)} \chf{\ind^{\kkSch{L}{x_0}}_{\kkSch{Q'}{x_0}} \fais{L}_{\bar\theta}}
=  \Ind^{\kSch{G}{x_0}(\kq)}_{\kSch{P}{x_0}(\kq)} \chf{ \NC{\RRSch{L}{x_0}} \fais{G}}.
\]
But, during the second step in the proof, we saw that $\fais{G}$ matched $\sigma$, which means that 
\[
\Ind^{\kSch{G}{x_0}(\kq)}_{\kSch{P}{x_0}(\kq)} \chf{ \NC{\RRSch{L}{x_0}} \fais{G}}
= \Ind^{\kSch{G}{x_0}(\kq)}_{\kSch{P}{x_0}(\kq)} \Trace\cRes^{\Ksch{L}(\Kq)}_{\RSch{L}{x_0}(\Rq)}\sigma.
\]
Finally, since $\sigma$ is cuspidal, and since $\pi = \Ind^{\Ksch{G}(\Kq)}_{\Ksch{P}(\Kq)} \sigma$, it follows  that
\[
\cRes^{\Ksch{G}(\Kq)}_{\RSch{G}{x_0}(\Rq)} \pi
= \cRes^{\Ksch{G}(\Kq)}_{\RSch{G}{x_0}(\Rq)} \Ind^{\Ksch{G}(\Kq)}_{\Ksch{P}(\Kq)} \sigma.
\]
From \cite[C.1.4]{V2}, we have 
\[
\cRes^{\Ksch{G}(\Kq)}_{\RSch{G}{x_0}(\Rq)} \Ind^{\Ksch{G}(\Kq)}_{\Ksch{P}(\Kq)} \sigma 
= \Ind^{\kSch{G}{x_0}(\kq)}_{\kSch{P}{x_0}(\kq)} \cRes^{\Ksch{L}(\Kq)}_{\RSch{L}{x_0}(\Rq)}\sigma.
\]
Thus,
\[
\Ind^{\kSch{G}{x_0}(\kq)}_{\kSch{P}{x_0}(\kq)} \Trace\cRes^{\Ksch{L}(\Kq)}_{\RSch{L}{x_0}(\Rq)}\sigma 
= \Trace \cRes^{\Ksch{G}(\Kq)}_{\RSch{G}{x_0}(\Rq)} \pi.
\]
This concludes the proof of \eqref{equation: representations 4} and thus concludes the proof of Theorem~\ref{theorem: representations}.
\end{proof}

\begin{remark}\label{remark: decomposition}
The generalised principal series representation $\pi$ need not be irreducible, and the matching sheaf $(-1)^N\fais{F}$ need not be a character sheaf (although $\fais{F}$ is a perverse sheaf and a finite direct sum of character sheaves); however, as we will show elsewhere, when $\pi$ is irreducible, $\fais{F}$ is indeed a character sheaf.
\end{remark}

\bibliographystyle{amsalpha}

\begin{thebibliography}{Lus85/86}

\bibitem[BBD]{BBD} Alexander BEILINSON, Joseph BERNSTEIN \& Pierre DELIGNE, {\em Faisceaux Pervers}. Analysis and topology on singular spaces, I (Luminy, 1981), 5--171, Ast\'erisque {\bf 100}, Soc. Math. France, Paris, 1982.

\bibitem[Ber84]{Ber} Joseph BERNSTEIN, Le ``centre'' de Bernstein, in {\it Repr\'esentations des groupes r\'eductifs sur un corps local}, by Bernstein, Kazhdan, Vigneras, Hermann, 1984.

\bibitem[BGR90]{BGR} Siegfried BOSCH, Werner L\"{U}TKEBOHMERT, Michel RAYNAUD, {\it N\'eron Models}, Ergebnisse der Mathematik und ihrer Grenzgebiete 3. Folge --- Band 21, Springer-Verlag, 1990.

\bibitem[BT84]{BT2} Fran\c{c}ois BRUHAT \& Jacques TITS, \emph{Groupes r\'eductifs sur un corps local. II. Sch\'emas en groupes. Existence d'une donn\'ee radicielle valu\'ee}, Inst. Hautes \'Etudes Sci. Publ. Math. No. 60 (1984), 197--376.

\bibitem[Cun00]{C} Clifton CUNNINGHAM, \emph{Characters of depth zero supercuspidal representations of the rank-2 symplectic group}, Canad. J. Math. 52 (2000), 306--331. 

\bibitem[Del80]{D2} Pierre DELIGNE, {\em La Conjecture de Weil II}, Publ. Math. IHES {\bf 52} (1980), pp. 137--252.

\bibitem[DL76]{DL} Pierre DELIGNE and Georges LUSZTIG, {\em Representations of reductive groups over finite fields}, Annals of Mathematics, {\bf 103} (1976), 103--161. 

\bibitem[Eke90]{Ek} Torsten EKEDAHL, \emph{On the adic formalism}, in The Grothendieck Festschrift Vol. II, Birkh\"{a}user Verlag, 1990.

\bibitem[Ill06]{Ill} Luc ILLUSIE, \emph{Vanishing cycles over general bases after P. Deligne, O. Gabber, G. Laumon and F. Orgogozo}, 8 March 2006.

\bibitem[Lan96]{Lan1} Erasmus LANDVOGT, {\em A compactification of the Bruhat-Tits building}, Lecture Notes in Mathematics, 1619. Springer-Verlag, Berlin, 1996.

\bibitem[Lan00]{Lan2} Erasmus LANDVOGT, {\em Some functorial properties of the Bruhat-Tits building}, J. Reine Angew. Math. {\bf 518} (2000),
pp.~213--241.

\bibitem[Lau87]{Lau} G\'erard LAUMON, {\em Transformation de Fourier, constantes d'\'equations fonctionnelles et conjecture de Weil}, Publ. Math., Inst. Hautes \'Etud. Sci. {\bf 65}  (1987), pp.~131--210.

\bibitem[Lus84]{L0} George LUSZTIG, {\em Intersection cohomology on a reductive group}, Invent. Math. {\bf 75}, pp.~205--272 (1984).

\bibitem[Lus85/86]{CS} George LUSZTIG, {\em Character Sheaves I}, {Advances in Mathematics} {\bf 56} (1985), {pp.~193--297}; {\em Character Sheaves II}, {Advances in Mathematics} {\bf 57} (1985), {pp.~226--265}; {\em Character Sheaves III}, {Advances in Mathematics} {\bf 57} (1986), {pp.~266--315}; {\em Character Sheaves IV}, {Advances in Mathematics} {\bf 59} (1986), {pp.~1--63}; {\em Character Sheaves V}, {Advances in Mathematics} {\bf 61} (1986), {pp.~103--155}.

\bibitem[Lus95]{Ldual} G. LUSZTIG, ClassiÞcation of unipotent representations of simple p-adic groups, Int. Math. 
Res. Not. (1995), 517--589.

\bibitem[MS89]{MS} Jean-G\'erard~MARS \& Tonny~SPRINGER, \emph{Character Sheaves}. Orbites unipotentes et repr\'esentations, III. Ast\'erisque No. {\bf 173--174} (1989), 9, {pp. 111--198}.

\bibitem[Pra01]{P} Gopal PRASAD, {\em Galois-fixed points in the Bruhat-Tits building of a reductive group}. Bulletin Soc. Math. France {\bf 129} (2001), {pp.~169--174}.

\bibitem[SGA4]{SGA4} {\em S\'eminaire de G\'eom\'etrie Alg\'ebrique du Bois-Marie 1963-1964: Th\'eorie des Topos et Cohomologie \'Etale des Sch\'emas (SGA4)}, dirig\'e par M.~ARTIN, A.~GROTHENDIECK et J~.L.~VERDIER, avec la collaboration de N.~Bourbaki, P.~Deligne, S.~Saint-Donat, Springer-Verlag, Lecture Notes in Mathematics 269, 270 and 305 (1972-73).

\bibitem[SGA4.5]{SGA4.5}  {\em S\'eminaire de G\'eom\'etrie Alg\'ebrique du Bois-Marie: Cohomologie \'Etale (SGA $4\frac{1}{2}$)} par Pierre DELIGNE avec la collaboration de J.~F.~Boutot, A.~Grothendieck, L.~Illusie et J.~L.~Verdier, Springer-Verlag, Lecture Notes in Mathematics 569 (1977).

\bibitem[SGA5]{SGA5} {\em S\'eminaire de G\'eom\'etrie Alg\'ebrique du Bois-Marie 1965-66: Cohomologie $\ell$-adique et Fonctions L (SGA $5$)}, dirig\'e par A.~GROTHENDIECK, avec la collaboration de I.~Bucur, C.~Houzel, L.~Illusie, J.-P.~Jouanolou et J.-P.~Serre, Springer-Verlag, Lecture Notes in Mathematics 586.

\bibitem[SGA7]{SGA7} {\em S\'eminaire de G\'eom\'etrie Alg\'ebrique du Bois-Marie 1967-1969: Groupes de Monodromie en G\'eom\'etrie Alg\'ebrique (SGA $7$)}, par P. DELIGNE et N. KATZ.

\bibitem[SS97]{SS} Peter SCHNEIDER \& Ulrich STUHLER, {\em Representation theory and sheaves on the Bruhat-Tits building}, Inst. Hautes \'Etudes Sci. Publ. Math. {\bf 85}, pp.~97--191, 1997.

\bibitem[Sho95]{shoji} Toshiaki SHOJI, {\em Character sheaves and almost characters of reductive groups}, Advances in Math.,
{\bf 111} (1995), 244-313.

\bibitem[Sri79]{srinivasan} Bhama SRINIVASAN, {\em Representations of finite Chevalley groups}, Lecture Notes in Mathematics 764,
Springer, Berlin, 1979.
 
\bibitem[Vig03]{V2} Marie-France VIGNERAS, {\em Schur algebras of reductive p-adic groups I}, Duke Math. J. {\bf 116}, No. 1 (2003), pp. 35--75.

\bibitem[Wal01]{W} Jean-Loup WALSPURGER, {\em Int\'egrales orbitales nilpotentes et endoscopie pour les groupes classiques non ramifi\'es}, Ast\'erisque {\bf 269}, 2001.

\bibitem[Yu02]{Yu} Jiu-Kang YU, {\it Smooth models associated to concave functions in Bruhat-Tits theory}, preprint, 13 December 2002.

\end{thebibliography}

\end{document}